\documentclass[11pt,a4paper]{amsart}

\usepackage[left=3cm, top=3cm,bottom=3cm,right=3cm]{geometry}

\usepackage{mathtools,amssymb,amsthm,mathrsfs,calc,graphicx,stmaryrd,xcolor,dsfont,pgfplots,bbm,float,array,epsfig,color}
\usepackage[numbers,square]{natbib}
\usepackage[british]{babel}
\usepackage{amsfonts,amsmath,amsthm,amssymb}
\usepackage{verbatim}
\usepackage{multirow}
\usepackage{tikz}
\usepackage{bbm}

\numberwithin{equation}{section}
\numberwithin{figure}{section}

\newtheorem{thm}{Theorem}[section]
\newtheorem{lem}[thm]{Lemma}
\newtheorem{cor}[thm]{Corollary}
\newtheorem{prop}[thm]{Proposition}

\theoremstyle{remark}
\newtheorem{rem}[thm]{Remark}

\theoremstyle{definition}
\newtheorem{ex}[thm]{Example}

\numberwithin{equation}{section}

\def\P{{\rm P}}
\def\Z{{\mathbb{Z}}}

\def\R{{\mathbb{R}}}

\def\F{{\mathcal{F}}}

\newcommand{\PP}{\mathbb{P}}
\newcommand{\E}{{\rm E}}
\newcommand{\EE}{\mathbb{E}}

\newcommand{\G}{\mathcal{G}}
\newcommand{\X}{\mathcal{X}}
\newcommand{\M}{\mathcal{M}}

\newcommand{\Po}{{\rm P}_{\omega}}
\newcommand{\Eo}{{\rm E}_{\omega}}

\newcommand{\NN}{\mathbb{N}}
\newcommand{\N}{\mathbb{N}}
\newcommand{\T}{\mathbb{T}}
\newcommand{\ZZ}{\mathbb{Z}}

\newcommand{\RR}{\mathbb{R}}

\newcommand{\var}{{\rm Var}}

\newcommand{\ud}{\mathrm{d}}
\newcommand{\eps}{\varepsilon}

\renewcommand{\l}{\lambda}
\DeclareMathOperator{\1}{\mathbbm{1}}

\newcommand{\bj}{\mathbf{j}}
\newcommand{\bL}{\mathbf{L}}

\definecolor{ala}{RGB}{20,100,30}
\definecolor{p}{RGB}{100,1,100}

\usepackage{todonotes}

\begin{document}
	\title[Weak quenched limit theorems for RWSRE]{Weak quenched limit theorems for a random walk in a sparse random environment}
	\author[D. Buraczewski, P. Dyszewski and A. Kołodziejska]{Dariusz Buraczewski, Piotr Dyszewski and Alicja Kołodziejska}
	\date{}

	\begin{abstract}
		We study the quenched behaviour of a perturbed version of the simple symmetric random walk on the set of integers. The random walker moves symmetrically with an exception of some  randomly chosen sites where we
		impose a random drift. We show that if the gaps between the marked sites are i.i.d. and regularly varying with a sufficiently small index, then there is no strong quenched limit laws for the position of the random walker.
		As a consequence we study the quenched limit laws in the context of weak convergence of random measures.
	\end{abstract}


\maketitle

\section{Introduction}\label{sec:intro}

One of the most classical and well-understood random processes is the
simple symmetric random walk (SRW) on the set of integers, where the particle starting
at zero every unit time moves with probability $1/2$ to one of its neighbours.
This process is a time and space homogeneous Markov chain, that is its increments are
independent of the past and the transitions do not depend on time and the current
position of the process.  In many cases, the homogeneity of the environment reduces the applicability
of the process. In numerous applied models  some kind of obstacles can appear like impurities,
fluctuations, etc. Thus, it is natural to express such irregularities as a random environment
and it is well known that even  small perturbations of the environment
affect properties of the random process. In 1981 Harrison and Shepp \cite{Harrison:Shepp:1981}
described the behaviour of the SRW in a slightly disturbed environment, replacing only the probability of passing from 0 to 1 by some fixed $p\in (0,1)$.  They observed that the scaling limit is not the Brownian motion, but the skew Brownian motion.

We intend to study random walks in a randomly perturbed environment.  Our main results concern
the so-called random walk in a sparse random environment (RWSRE) introduced in \cite{matzavinos:2016:random}, in which homogeneity of an environment is
perturbed only on a sparse subset of $\Z$. More precisely, first we choose randomly
a subset of integers  marked by the positions of a standard random walk
with positive integer jumps and next we impose a random drift at the chosen sites.
The present paper can be viewed as a continuation of the recent publications \cite{matzavinos:2016:random,buraczewski:2019:random,buraczewski:2020:random},
where annealed limit theorems were described.
These annelad-type results do not settle however the question if the environment alone is sufficient
to determine the distributional behaviour of the process with high certainty.
Here we froze the environment and we are interested
in limit behaviour of the random process in the quenched settings.
As we show in the present article, even in a very diluted random environment the fluctuations of the
random perturbation of the medium affect the conditional distribution of the random walker.

The model RWSRE we consider here can be viewed as an interpolation between SRW and the one
suggested in the seventies by Solomon~\cite{solomon1975random} called  a one dimensional random walk in random environment (RWRE),
where all the sites were associated with random i.i.d. weights $\{\omega_i\}$ describing the probability
of passing to the right neighbour.
It quickly became clear that the additional environmental noise in the system has a
 significant impact  on the behaviour of the model. In fact, the answers to a variety of questions about the model like limit theorem~\cite{kesten1975limit} and large
	deviations~\cite{dembo1996tail, buraczewski2018precise} %, and the asymptotic behaviour of some related characteristics,
	are given only in terms of the environment marginalizing the
	impact of the random motion of the process.

%Lately Matzavinos et al.~\cite{matzavinos:2016:random} proposed a variation on the Solomons model in which the environment is randomly
%	perturbed only in random sites. The distance between two marked sites in controlled by a probability distribution. It turns out that when the aforementioned distribution is %heavy tailed the
%	model reveals an interplay between the contribution of the medium and the random walker~\cite{buraczewski:2020:random}.

\subsection{General setting}

  	To define our model let $\Omega=(0,1)^\Z$ be the set of all possible configurations of the~environment equipped with the corresponding cylindrical $\sigma$-algebra $\F$ and a probability
  	measure $\P$. A random element $\omega=(\omega_n)_{n\in\Z}$ of $(\Omega, \F)$ distributed according  to $\P$ is called a {\it random environment}.
  	Each element $\omega$ of $\Omega$ and integer $x$ gives a rise to a~probability measure $\P_\omega^x$ on the set $\mathcal{X} = \Z^{\N_0}$ with the cylindrical  $\sigma$-algebra
  	$\mathcal{G}$ such that $\Po^x[X_0=x]=1$ and
	\begin{equation*}
		\P_\omega^x  \left[X_{n+1}=j| X_n = i \right] = \left\{
		\begin{array}{cl}
			\omega_i, & \mbox{if } j=i+1,\\
			1- \omega_i,\ & \mbox{if } j=i-1,\\
			0, & \mbox{otherwise,}
		\end{array}\right.
	\end{equation*}
	where $ X = (X_n)_{n \in \N_0} \in \mathcal{X}$.  One sees that under $\P_\omega^x$, $X$ forms a nearest neighbour random walk which
%we will call 	a \textit{random walk in random environment}.
%	Clearly, under $\P_\omega^x$, $X$
is a time-homogeneous Markov chain on $\Z$ and it is called a \textit{random walk in random environment.}
 The randomness of the environment $\omega$ influences significantly various properties of $X$.
  	In view of this, it is natural to investigate the behaviour of $X$ under the annealed measure $\PP^x= \int \P_\omega^x \P(\ud\omega)$ which is defined as the unique
  	probability measure on $(\Omega\times \X, \F\otimes \G)$ satisfying
	\begin{equation*}
		\PP^x[F\times G] = \int_F \P_{\omega}^x[G] \: \P(\ud \omega),\quad F\in \F,\quad G\in \G.\
	\end{equation*}
	In the sequel we will write $\Po=\Po^0$ and $\PP=\PP^0$.
  	It turns out that, in general, under the annealed probability $X$ is no longer  a Markov chain, because
  %. In fact $X$
  it usually exhibits a long range dependence.
  %and as a consequence 	one uses more involved tools and methods than for a random walks on $\Z$.

  \medskip

  	We are interested in limit theorems for $X_n$ as $\to \infty$,  however in this paper
  %, for reasons that will become apparent later,
%  it  will be more convenient to
we discuss the
  	asymptotic behaviour of the corresponding sequence of first passage times $T = (T_n)_{n\in \NN}$, that is 
  	\begin{equation}\label{eq:1:FirstPassage}
    		T_n = \inf \{ k \in \NN \: : \: X_k = n\}.
  	\end{equation}
%	By a classical inversion technique \'a la Feller~\cite{feller:1949:fluctuation} one can translate the limit results for $T=(T_n)_{n \in \NN}$ to $X$. {\color{ala} \tt Czego nie umiemy zrobic.}
	We will study the distribution of $T_n$ in the quenched setting which means
	that we will investigate the behaviour of
	\begin{equation*}
		\mu_{n, \omega}(\cdot) =  \P_\omega \left[ (T_n-b_n)/a_n \in \cdot \: \right]
	\end{equation*}	
	for suitable choices of sequences $(a_n)_{n \in \NN}$ and $(b_n)_{n \in \NN}$ possibly depending on $\omega$. In the present setting
 $\mu_{n}$, defined by $\mu_n(\omega) = \mu_{n,\omega}$,  becomes a random element of
	$\mathcal{M}_1$, the space of probability measures on $(\RR, \mathcal{B}or(\RR))$, where
	$\mathcal{B}or(\RR)$ stands for the Borel $\sigma$-algebra. $\mathcal{M}_1$  equipped in the Prokhorov distance is a complete, separable metric space. One can distinguish two
	types of limiting behaviour of $(\mu_n)_{n \in \NN}$. We will say that a strong quenched limit theorem for $T$ holds if $\mu_n \to \mu$ almost surely in  $\mathcal{M}_1$, that is for $\P$ a.s. $\omega$ the sequence of measures $\{\mu_{n,\omega}\}$ converges weakly to $\mu$,
  and say that a
	weak quenched limit law for $T$ holds if $\mu_n \Rightarrow \mu$ in $\mathcal{M}_1$.
	Here and in the sequel $\Rightarrow$ denotes weak convergence. \medskip
	
	We will now discuss different choices of the probability $\P$ which is the distribution of the environment. To keep the introduction brief we will limit the discussion to the i.i.d. random environment, which is the most classical choice for
	$\P$, and the sparse random environment which we will study in depth in the sequel.

\subsection{Independent identically distributed environment}
	One of the simplest and most studied choices of the environmental distribution $\P$  is {\it random walk in i.i.d.~random environment,}  corresponding to a product measure, under which $\omega = (\omega_n)_{n\in\Z}$
	forms a collection of independent, identically distributed (i.i.d.) random variables. In their seminal work Kesten et al.~\cite{kesten1975limit} used the following link between walks and random
	trees~\cite{harris1952first}: the one-dimensional distributions of $T$ are connected to a branching process in random environment with immigration and a reproduction
	law with the mean distributed as $(1-\omega_0)/\omega_0$. This observation later leads to a conclusion that $T$ lies in the domain of attraction of an $\alpha$-stable distribution, where
	$\E[ \omega_0^{-\alpha}(1-\omega_0)^\alpha]=1$, provided that such  $\alpha\in(0,2)$  exists  (see Figure \ref{fig:one}).
	\begin{figure}\label{fig:one}
		\includegraphics[scale=0.4]{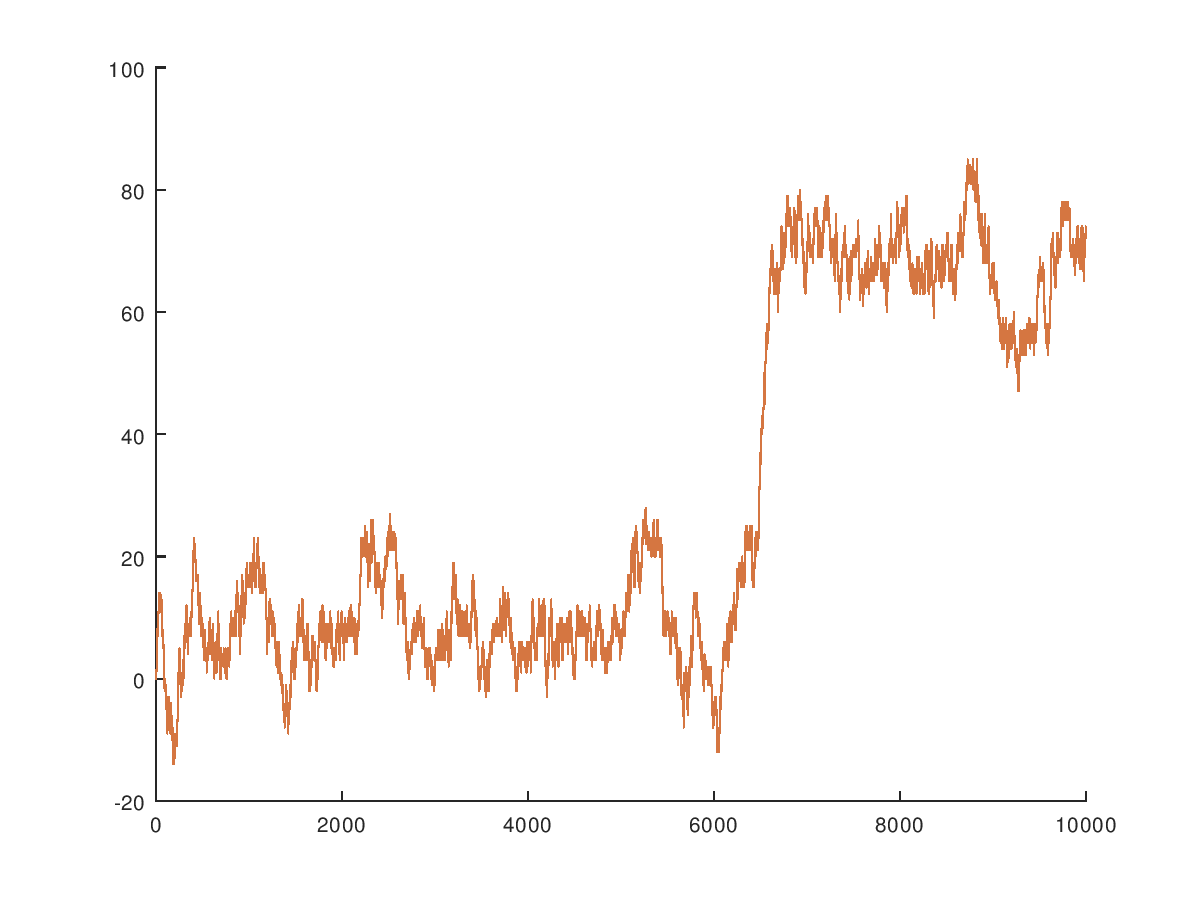}
		\centering
		\caption{Random walk in i.i.d.~random environment for $\omega_0 = 1/3$ with probability $1/3$, $\omega_0=3/4$ with probability $2/3$  and $\alpha \approx 1,35$   .}
	\end{figure}
  After a close examination of the main results of Kesten et al.~\cite{kesten1975limit} it transpires
	that the centering and scaling are determined by the distribution of $(1-\omega_0)/\omega_0$, which means that the behaviour of the walker does not affect the limiting behaviour in a
	significant way. In turn, to understand the random motion, one is led to investigate the behaviour of $T$ under $\P_\omega$.  If $\alpha>2$, then a strong quenched limit  theorem~\cite{alili1999asymptotic,goldsheid2007simple}  of the form
	\begin{equation*}
		\lim_{n \to \infty}\P_\omega \left[ (T_n-\E_\omega[T_n])/(\sigma \sqrt{n}) \in \ud x \: \right] = e^{-x^2/2} \ud x/\sqrt{2\pi}
	\end{equation*}
	holds almost surely in $\mathcal{M}_1$, where $\sigma^2 = \E [ \var_\omega [T_1]] <\infty$. As seen from the results
	in~\cite{peterson2009quenched, peterson2009quenched2} there is no strong quenched limit theorem for $T$ in the case $\alpha<2$. Indeed it turns out that for $\alpha<2$ one can find
	different strong quenched limits for $T$ along different sequences. This in turn leads to the analysis of $T$ in the weak quenched setting, that is weak limits of $\mu_{n}$. Consider first the mapping
	$ H: \M_p \to \M_1$ given as follows: for a point process $\zeta = \sum_{i\ge 1}\delta_{x_i}$, where $\{x_i\}_{i \in \NN}$ is an arbitrary enumeration of the points, define
  	\begin{equation*}
    		H(\zeta)(\cdot) = \left\{
      		\begin{array}{cc}
        		\PP\big[ \sum_{i\ge 1} x_i(\tau_i-1)\in \cdot \big], \ \  & \sum_{i\ge 1} x_i^2<\infty, \\
        		\delta_0(\cdot), & \mbox{otherwise,}
      		\end{array}
    		\right.
  	\end{equation*}
	where $\{\tau_i\}_{i \in \NN}$ is a sequence of i.i.d. mean one exponential random variables. Then the main result  of
\cite{dolgopyat:goldsheid, enriquez:i:spolka,peterson:2013:weak}
states that for $\alpha<2$,
	\begin{equation*}
		\P_\omega \big[ n^{-1/\alpha}(T_n - \E_\omega T_n) \in \cdot \: \big] \Rightarrow  H(N)
	\end{equation*}
	in $\M_1$, where $N$ is a Poisson point process on $(0, \infty)$ with intensity $c_N x^{-\alpha -1} \ud x$ for some constant $c_N>0$.

\subsection{Sparse random environment}

  	We now specify the object of interest in the present paper. We will work under a choice of environmental probability $\P$ for which the random walk $X$ will move
  	symmetrically except some randomly marked points where we impose a random drift. The marked sites will be distributed according to a two-sided random walk.
	Denote by $((\xi_k,\lambda_k))_{k\in \Z}$ a sequence of independent copies of a
  	random vector $(\xi,\lambda)$, where $\lambda \in (0,1)$ and $\xi \in \N$, $\P$-a.s. Considering the aforementioned two-sided random walk $S = (S_n)_{n \in \ZZ}$ given via
  	\begin{equation*}
      		S_n = \left\{ \begin{array}{cl} \sum_{k=1}^n \xi_k, & \mbox{if } n>0, \\
          		0, & \mbox{if } n=0, \\
          		-\sum_{k=n+1}^{ 0} \xi_k, & \mbox{if } n<0,
        		\end{array} \right.
  	\end{equation*}
  	we define a random environment $\omega = (\omega_n)_{n \in \Z} \in \Omega$ given by
  	\begin{equation}\label{eq:sparse}
    		\omega_n = \left\{
      			\begin{array}{ll}
        			\l_{k}, \ & \mbox{if $n=S_k$ for some }k\in\Z,\\
        			1/2, \ & \mbox{otherwise.}
      			\end{array}\right.
  	\end{equation}
  	The sequence $S$ determines the marked sites in which the random drifts $2\lambda_{k}-1$ are placed. Since for the unmarked sites $n$
  	(that is, for most of sites) the probabilities of jumping to the right are deterministic and equal to $\omega_n=1/2$, it is natural to call $\omega$ a {\it sparse random environment}.
  	Following~\cite{matzavinos:2016:random} we use the term {\it random walk in sparse random environment} (RWSRE)
  	for $X$ as defined above with $\omega$ being a~sparse random environment.

  	\begin{ex}
    		In the case when $\P[\xi=1]=1$ random walk in sparse random environment is equivalent to a random walk in i.i.d. environment.
  	\end{ex}

  	\begin{ex}
    		Suppose that $\xi$ is independent of $\lambda$ and has a geometric distribution $\P[\xi=k] = a (1-a)^{k-1}$, $k\geq 1$ for some $a \in (0,1)$.
    		Then the sparse random environment given in~\eqref{eq:sparse} is equivalent to an i.i.d. environment with $\omega_0$ distributed as
    		$\P[\omega_0 \in \cdot] = a\P[\lambda \in \cdot] + (1-a) \delta_{1/2}(\cdot)$.
  	\end{ex}

  	Random walk in a sparse random environment was studied in detail in the annealed setting in~\cite{matzavinos:2016:random, buraczewski:2019:random, buraczewski:2020:random}.
	In~\cite{matzavinos:2016:random} the authors address the question of transience and recurrence of RWSRE and prove a strong law of large numbers and some distributional limit theorems for
	$X$. As in the case of i.i.d. random environment, the fraction
	\begin{equation*}
		\rho = \frac{1-\lambda}{\lambda}
	\end{equation*}
	appears naturally in the description of the asymptotic behaviour of the random walk.
	According to~\cite[Theorem 3.1]{matzavinos:2016:random}, $X$ is
	$\P$-a.s.~transient to $+\infty$ if
	\begin{equation}\label{eq:right_transience}
		\E \log \rho \in [-\infty, 0) \quad \text{and}\quad \E\log \xi<\infty.
	\end{equation}
	Note that the first condition in \eqref{eq:right_transience} excludes the degenerate case $\rho = 1$ a.s.\ in which $X$ is a simple random walk.
	Under \eqref{eq:right_transience}, the RWSRE also satisfies a strong law of large numbers, that is,
	\begin{equation}\label{eq:LLN}
		T_n/ n \to  1/v \quad \PP-a.s.
	\end{equation}
	where
	\begin{equation*}
		v = \left\{ \begin{array}{cl}   \frac{(1 - \E \rho) \E \xi}{ (1-\E \rho) \E\xi^2 + 2 \E \rho \xi\E \xi} & \mbox{if  $\E \rho <1$,
			$\E \rho\xi<\infty$ and $\E \xi^2<\infty$}, \\
			0 & \mbox{otherwise,}     \end{array} \right.
	\end{equation*}
	see Theorem 3.3  in \cite{matzavinos:2016:random} and Proposition 2.1 in~\cite{buraczewski:2019:random}. We note right away that conditions  present in~\eqref{eq:right_transience}
	are satisfied under the conditions
	of our main results. Thus, the random walks in a sparse random
	environment that we treat here are transient to the right.\medskip

	The asymptotic behaviour of $T$ is controlled by two ingredients. The first one, similarly as in the case of i.i.d. environment, is $\alpha >0$ such that
	\begin{equation}\label{eq:1:KestenSparse}
		\E \left[ \rho^\alpha \right]=1.
	\end{equation}
	The parameter $\alpha>0$, if it exists, is used to quantify the effect that the random transition probabilities $\lambda_k$'s have on the asymptotic behaviour of the random walker. The second ingredient is the tail behaviour of
	$\xi$, that is the asymptotic of $\P[\xi >t ]$ as
	$t \to \infty$. If $\E[\xi^4] <\infty$, then with respect to the annealed probability $T$ is in the domain
	of attraction of an $\alpha$-stable distribution~\cite[Theorem 2.2]{buraczewski:2019:random}
	with the exact same behaviour as one observes in the case of i.i.d. environment.
 	New phenomena appear if $\xi$ has a regularly varying tail with index $-\beta$ for
	$\beta \in (0,4)$, i.e. as $t \to \infty$,
	\begin{equation*}
		\P[\xi>t] \sim t^{-\beta}\ell(t)
	\end{equation*}
	for some function $\ell \colon \RR \to \RR$ slowly varying at infinity. Here and in the rest of the article we write $f(t) \sim g(t)$ for two functions $f, g \in\RR \to \RR$ whenever $f(t)/ g(t) \to 1$ as $t \to \infty$.
	Recall that a function $\ell$ is slowly varying at infinity if $\ell(c t) \sim \ell(t)$ as $t \to \infty$ for any constant $c>0$.
	It transpires that if the tail of $\xi$ is regularly varying with $\beta \in (0,4)$ with $\E [\xi]<\infty$, then with respect to the annealed probability $T$ lies in the domain of attraction of
	$\gamma$-stable distribution with $\gamma = \min \{ \alpha, \beta/2 \}$, see~\cite{buraczewski:2019:random}.
	\begin{figure}\label{fig:two}
		\includegraphics[scale=0.4]{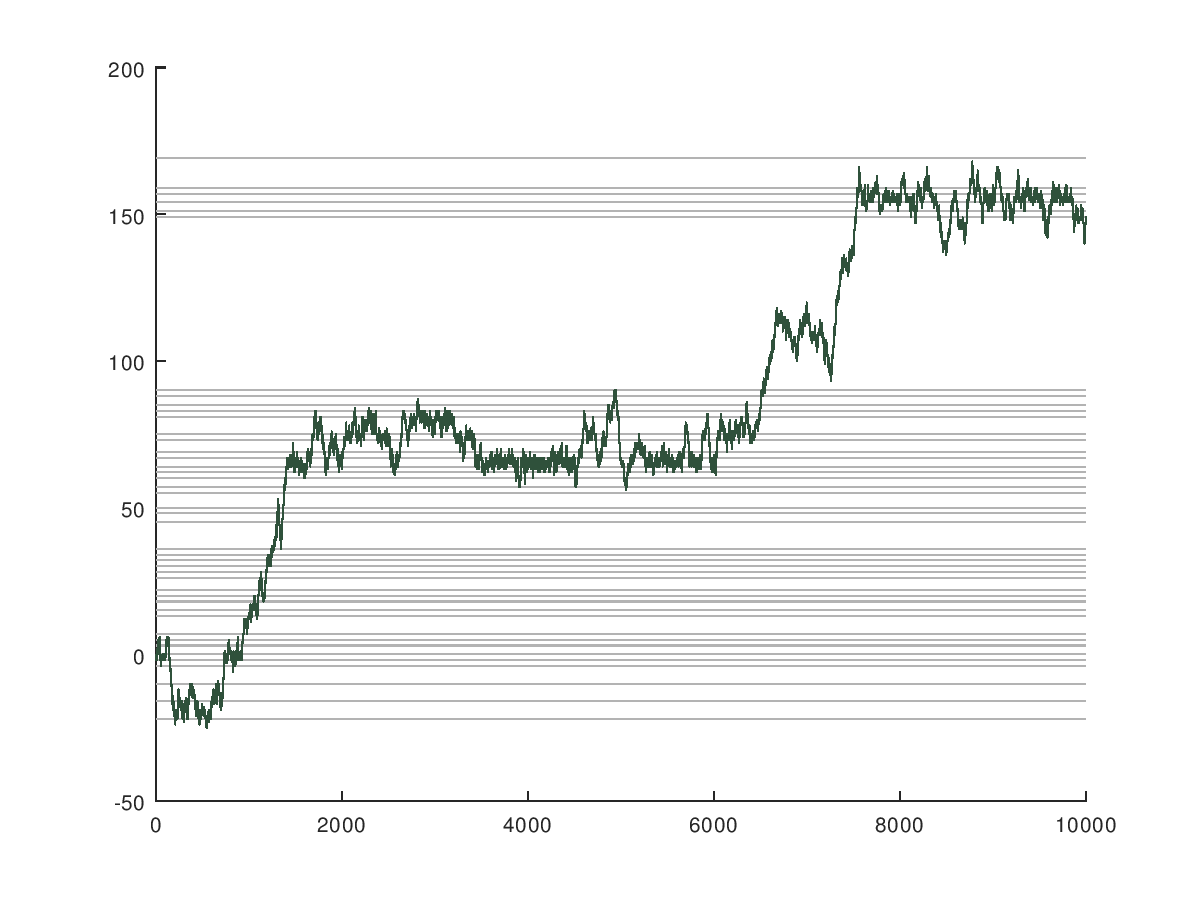}\includegraphics[scale=0.4]{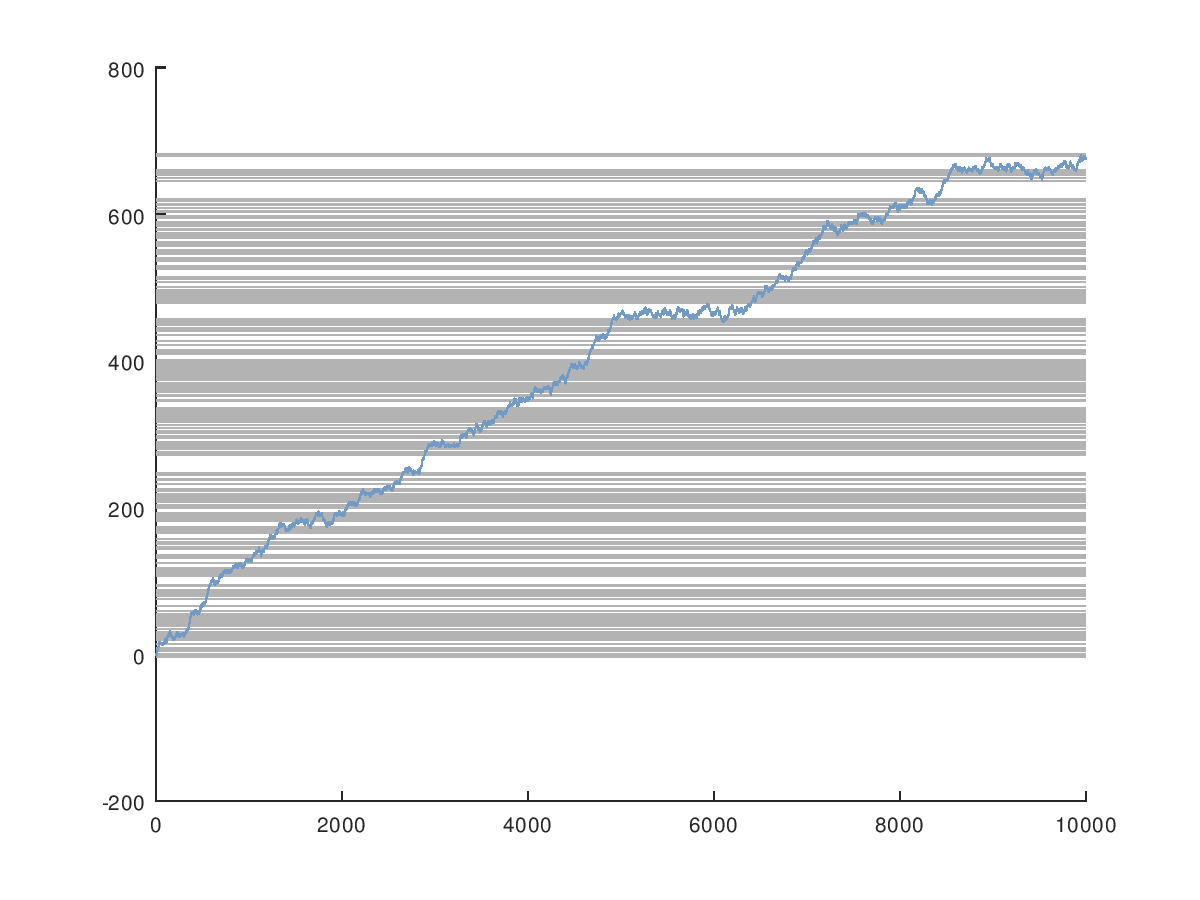}
		\centering
		\caption{RWSRE: $\beta =1.2$ and $\alpha \approx 0.52$ (left) and $\beta =1.2$ and $\alpha \approx 1.85$ (right). The grey horizontal lines indicate the marked sites}
	\end{figure}
	
	For small values of $\beta$ one sees an interplay between the contribution of the sparse random environment and the random movement of the process in the unmarked sites.
	To state this result take $\vartheta$ to be a non-negative random variable with the Laplace transform
	\begin{equation}\label{eq:mm1}
  		\E \left[ e^{-s \vartheta} \right] = \frac{1}{{\rm cosh}(\sqrt{s })}, \quad s>0.
	\end{equation}
	Note that $2\vartheta$ is equal in distribution to the exit time of the one-dimensional Brownian motion from the interval $[-1,1]$, see~\cite[Proposition II.3.7]{revuz:2014:continuous}.
	Next consider a measure $\eta$ on $\mathbb{K}=[0, \infty]^2\setminus\{(0,0)\}$ given via
	\begin{equation*}
		\eta( \{ (v,u) \in \mathbb{K} \: : \: u>x_1 \mbox{ or } v>x_2\}) = x_1^{-\beta} +
			\EE[\vartheta^{\beta/2}] x^{-\beta/2}_2- \EE[\min \{ x_1^{-\beta}, \vartheta^{\beta/2} x_2^{-\beta/2}\}]
	\end{equation*}
	for $x_1, x_2>0$. Now let $N = \sum_{k} \delta_{(t_k, \bj_k )}$ be a Poisson point process on $[0, \infty) \times \mathbb{K}$ with intensity ${\rm LEB} \otimes \eta$, where ${\rm LEB}$
	stands for the one-dimensional Lebesgue measure.
	Under mild integrability assumptions, see~\cite[Lemma 6.4]{buraczewski:2020:random}, the integral
	\begin{equation*}
		\bL(t) = (L_1(t), L_2(t)) = \int_{[0,t] \times \mathbb{K}} \bj \: N(\ud s, \ud \bj), \qquad t \geq 0
	\end{equation*}
	converges and defines a two-dimensional non-stable L\'evy process with L\'evy measure $\eta$. Next consider the $\beta$-inverse subordinator
	\begin{equation*}
		L_1^{\leftarrow}(t) = \inf \{ s>0 \: : \: L_1(s) >t \}, \qquad t \geq 0.
	\end{equation*}
	Finally, if $\beta<2\alpha$ and $\beta \in (0,1)$, then under some additional mild integrability assumptions~\cite[Theorem 21]{buraczewski:2020:random}, with respect to the annealed probability
	\begin{equation*}
		T_n/ n^2 \Rightarrow 2 L_2(L_1^{\leftarrow}(1)^-) + 2\vartheta (1-L_1(L_1^{\leftarrow}(1)^- ))^2
	\end{equation*}
	weakly in $\RR$. The aim of the present article is to present a quenched version of this result.
	\begin{figure}\label{fig:one}
		\includegraphics[scale=0.4]{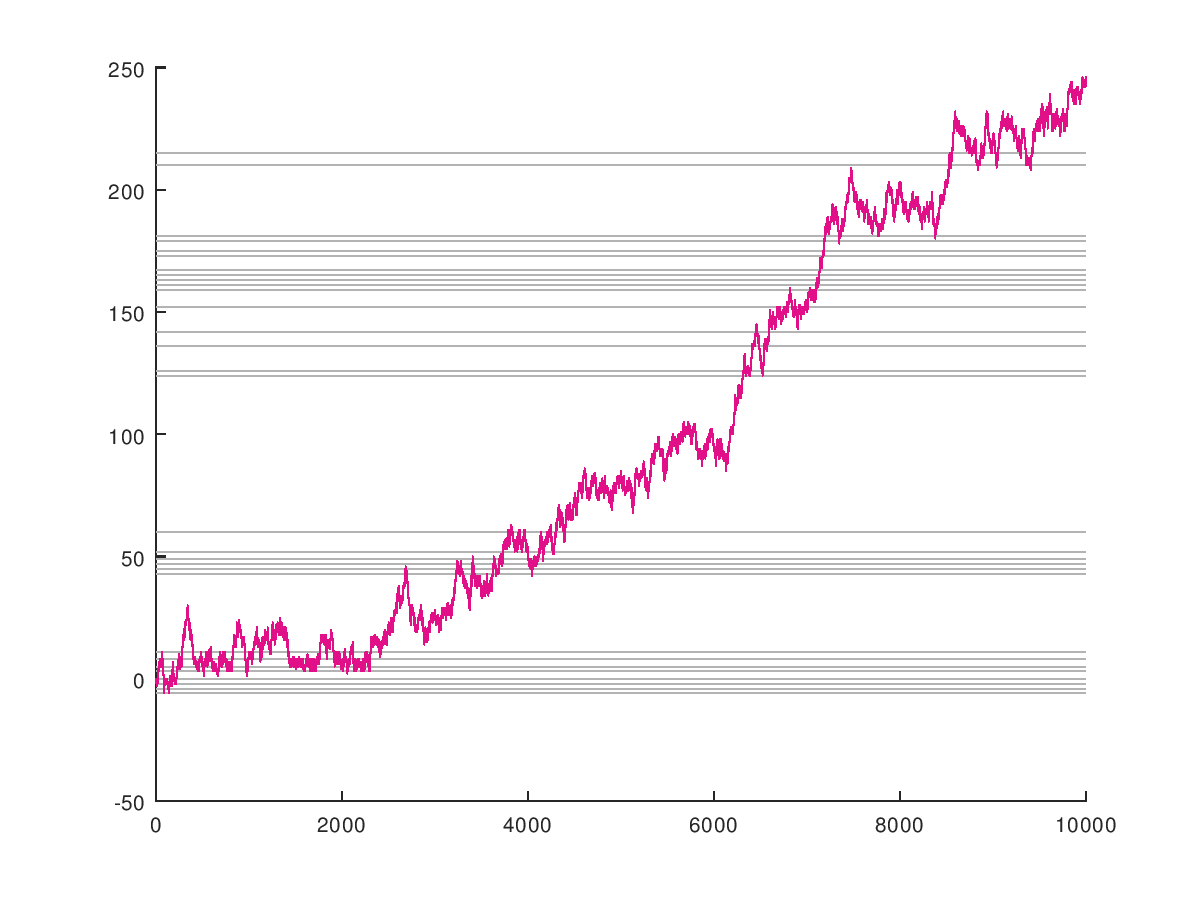}
		\centering
		\caption{RWSRE:  $\beta =0.8$ and $\alpha \approx 1.85$. The grey horizontal lines indicate the marked sites.}
	\end{figure}
	As we will see in our main theorem,  the terms $L_2(L_1^{\leftarrow}(1)-)$ and $L_1(L_1^{\leftarrow}(1-) )$ present on the right hand side can be viewed as the
	contribution of the environment, whereas $\vartheta$
	represents the contribution of the movement of the random walker in the unmarked sites that are close to $n$. For the full treatment of the annealed limit results, in particular the
	complementary case $ \beta\geq 2 \alpha$, we refer the reader to~\cite{buraczewski:2020:random}.\medskip

	The article is organised as follows: in Section~\ref{sec:weak} we give a precise description of our set-up and main results. In Section~\ref{sec:aux} we provide a preliminary analysis of the environment. The essential parts of the
	proof of our main results are in Sections~\ref{sec:absence} and~\ref{sec:proof} where we prove an absence of the strong quenched limits and prove weak quenched limits respectively.

\section{Weak quenched limit laws}\label{sec:weak}

	In this section we will present our main results. From this point we will consider only a sparse random environment given via~\eqref{eq:sparse}.
	We assume that
  	\begin{equation}\label{eq:2:assumption}
    		\P[ \xi >t] \sim t^{-\beta} \ell(t)
  	\end{equation}
	for some $\beta\in (0,4)$ and slowly varying $\ell$. We will focus on the case in which the asymptotic of the system is not determined solely by the environment and thus we will assume also that
  	\begin{equation}\label{eq:2:assumption2}
   		\E [\rho^{2\gamma}]<1, \qquad	 \E[\xi^\gamma \rho^{3\gamma}] < \infty
  	\end{equation}
  	for some parameter $\gamma \in (\beta/4, 1\wedge \beta)$. The first condition in~\eqref{eq:2:assumption2} guarantees that a part of the fluctuations of $T_n$ will come from the time that
	the process spends in the unmarked sites. The second condition is purely technical. Note that we do not assume that there exists $\alpha>0$ for which~\eqref{eq:1:KestenSparse} holds.\medskip

 	Our first result states that there is no quenched limit for $T_n$'s in the strong sense.  Take $(a_n)_{n \in \NN}$ to be any sequence of positive real numbers such that
	$$
		n\P[\xi > a_n] \to 1.
	$$
	Then, since the tail of $\xi$ is assumed to be regularly varying, the sequence $(a_n)_{n\in \NN}$ is also regularly varying with index $1/\beta$. That is for some slowly varying function $\ell_1$,
	$$
		a_n = n^{1/\beta}\ell_1(n).
	$$
	The sequence $(a_n)_{n \in \NN}$ will play the role of the scaling factor in our results. The first one shows an absence of strong quenched limit laws for $T$.
	%The choice of the scaling will become clear after we present statement of Theorem~\ref{thm:2:m1}.

 	\begin{thm}\label{mthm0}
  		Assume~\eqref{eq:right_transience}, \eqref{eq:2:assumption} and \eqref{eq:2:assumption2}.
  Then for $P$ almost every $\omega$ there are no sequences $\{A_n(\omega)\}_{n\in\N}$ and
  $\{C_n(\omega)\}_{n\in\N}$ such that the sequence of normalized random variables $(T_n - C_n(\omega))/A_n(\omega)$ converges
  in distribution (with respect to $P_{\omega}$) to a nontrivial random variable.
  \end{thm}
%	\begin{thm}\label{mthm0}
 % 		Assume~\eqref{eq:right_transience}, \eqref{eq:2:assumption} and \eqref{eq:2:assumption2} and consider
%		\begin{equation}\label{eq:2:kappa}
%			\kappa_n = \left\{ \begin{array}{cc} a_n^2 & \mbox{if} \: \E \xi < \infty \\ n^2 & \mbox{if} \:\E \xi =\infty \end{array} \right..
%		\end{equation}
%		Then $\omega-\P$-a.s. the sequence of probability distributions
 % 		\begin{equation}\label{eq:2:normh}
 % 			\P_\omega [ (T_n - \Eo T_n) /\kappa_n\in \: \cdot \: ]
  %		\end{equation}
 % 		is not tight.
%	\end{thm}

	Therefore, as in the case of i.i.d. environment, the asymptotic quenched behaviour of $T_n$'s ought to be expressed in terms of  weak quenched convergence. As it is the case for annealed
	limit theorem, one needs to distinguish between a moderately ($\E \xi <\infty$) and strongly ($\E \xi =\infty$) sparse random environment.

  	To describe the former take $\{\vartheta_j\}_{j \in \NN}$ to be a sequence of i.i.d. copies of $\vartheta$ distributed according to~\eqref{eq:mm1} and let $ G: \M_p \to \M_1$
	be given via
  	\begin{equation}\label{eq:G}
    		G(\zeta)(\cdot) = \left\{
      			\begin{array}{cc}
        			\P\left[ \sum_{i\ge 1} x_i(2\vartheta_i-1)\in \cdot \right], \ \  & \int x^2 \zeta (\ud x)<\infty, \\
        			\delta_0(\cdot) & \mbox{otherwise,}
      			\end{array}\right.
  	\end{equation}
  	for $\zeta = \sum_{i\ge 1}\delta_{x_i}$, where $\{x_i\}$ is an arbitrary enumeration of the point measure.

%{\red \tt trzeba poprawic normalizacje oraz miare intensywnosci}
	\begin{thm}\label{thm:2:m1}
		Assume~\eqref{eq:right_transience}, \eqref{eq:2:assumption} and~\eqref{eq:2:assumption2}. If $\E \xi <\infty$, then
		\begin{equation*}
			\P_\omega \left[ (T_n - \E_\omega T_n) / a_n^2 \in \cdot  \right] \Rightarrow   G(N)(\cdot)
    		\end{equation*}
    		in $\mathcal{M}_1$, where $N$ is a Poisson point process on $(0, \infty)$ with intensity  $ \beta x^{-\beta/2-1} \ud x/2 \E \xi$.
  	\end{thm}

  	Before we introduce the notation necessary to state our results in the strongly sparse random environment, we will first treat the critical case which is relatively simple to state.
	Denote	
	\begin{equation*}
		m_n = n\E \left[ \xi { \bf 1}_{\{ \xi \leq a_n \}} \right].
	\end{equation*}
	Note that by Karamata's theorem~\cite[Theorem 1.5.11]{bingham:1987:regular} the sequence $\{m_n\}_{n \in \NN}$ is regularly varying with index $1/\beta$. Furthermore $a_n = o(m_n)$ if $\beta=1$ and
	$a_n\sim (1-\beta)m_n$ if  $\beta<1$.
	Next let $\{ c_n\}_{n \in \NN}$ be the asymptotic inverse of $\{m_n\}_{n \in \NN}$, i.e. any increasing sequence of natural numbers such that
	\begin{equation*}
		\lim_{n \to \infty}c_{m_n}/n = \lim_{n \to \infty}m_{c_n}/n = 1.
	\end{equation*}
	By the properties of an asymptotic inversion of regularly varying sequences~\cite[Theorem 1.5.12]{bingham:1987:regular}, $c_n$ is well defined up to asymptotic equivalence and is regularly varying with index $\beta$.
	Finally, by the properties of the composition of regularly varying sequences $\{a_{c_n}\}_{n \in \NN}$ is regularly varying with index $1$ and $a_{c_n} = o(n)$ if $\beta =1$.

	\begin{thm}\label{thm:2:m2}
		Assume~\eqref{eq:right_transience}, \eqref{eq:2:assumption} and~\eqref{eq:2:assumption2}. If $\E \xi =\infty$ and $\beta=1$, then
		\begin{equation*}
			\P_\omega \left[ (T_n - \E_\omega T_n)/a_{c_n}^2 \in \cdot  \right] \Rightarrow   G(N)(\cdot)
    		\end{equation*}
    		in $\mathcal{M}_1$, where $N$ is a Poisson point process on $(0, \infty)$ with intensity $ x^{-3/2} \ud x/2$.
	\end{thm}

	The limiting random measures in Theorems~\ref{thm:2:m1} and~\ref{thm:2:m2} share some of the properties of their counterpart in the case of i.i.d. environment~\cite[Remark 1.5]{peterson:2013:weak}.
	Namely, using the superposition and scaling
	properties of Poisson point processes, one can directly show that for each $n\in \NN$ and $G, G_1, \ldots , G_n$ being i.i.d. copies of the limit random measure $G(N)$ in Theorem~\ref{thm:2:m1} or Theorem~\ref{thm:2:m2},
	\begin{equation}\label{eq:2:Gsimilar}
		G_1 * G_2* \ldots * G_n (\cdot) \stackrel{d}{=} G( \cdot / n^{2/\beta}).
	\end{equation}

	The statement of our results in the strongly sparse case needs some additional notation. As it is the case for the annealed results, it is most convenient to work in the framework of non-decreasing c\`adl\`ag functions rather
	than point processes.
	Denote by $\mathbb{D}^\uparrow$ the class of non-decreasing c\`adl\`ag functions $\RR_+ \to \RR_+$
  	and for $h \in \mathbb{D}^\uparrow$ consider
  	\begin{equation}\label{eq:2:upsilon}
  		\Upsilon(h) = \sup \{ h(t) \: : \: t \in \RR_+, \: h(t) \leq 1 \}.
  	\end{equation}
	Finally for $h \in \mathbb{D}^\uparrow$ denote by $ \{ x_k(h),  t_k (h)\}_k$ an arbitrary enumeration of jumps of $h$, that is $t_k=t_k(h) \in \RR_+$ for $k \in \NN$ are all points on the non-negative half line such that
	$h$ has a (left) discontinuity with jump of size $x_k(h) = h( t_k) - h(t_k^-)>0$ at $t_k$.
	Note that the random series $\sum_{k: h(t_k) \leq 1} x_k(h)^2 (2\vartheta_k-1)$ is convergent since it has an expected value bounded via $h(1)\EE|2\vartheta-1|$.
	Finally let $F \colon \mathbb{D}^\uparrow \to \mathcal{M}_1$ be given by
  	\begin{equation*}
    		F(h)(\cdot)= \P\left[ (1-\Upsilon(h))^2 (2\vartheta_0-1) + \sum_{k: h(t_k)\leq 1} x_k(h)^2(2\vartheta_k-1) \in  \: \cdot \: \right].
  	\end{equation*}
 	Note that if $h(t)=1$ for some $t$, then necessarily $\Upsilon (h)=1$.
 	 \begin{thm}\label{thm:2:m3}
    		Assume~\eqref{eq:right_transience}, \eqref{eq:2:assumption} and~\eqref{eq:2:assumption2}. If %$\E \xi =\infty$ and 
    $\beta \in (0,1)$, then
    		\begin{equation*}
      			\P_\omega \left[ (T_n-\E_\omega T_n)/n^{2} \in \cdot  \right] \Rightarrow F(L)(\cdot)
    		\end{equation*}
  		in $\mathcal{M}_1$, where $L$ is a $\beta$-stable L\'evy subordinator with L\'evy measure $\nu(x, +\infty) =x^{-\beta}$.
  	\end{thm}

	Interestingly the limit measure $F(L)$ does not enjoy a self-similarity property in the sense of~\eqref{eq:2:Gsimilar}. Namely, for any $a,b \in \RR$, $b>0$ the laws of
	\begin{equation*}
		F_1* F_2 (\cdot) \quad \mbox{and} \quad F(( \cdot -a)/b)
	\end{equation*}
	are different, where $F, F_1$ and $F_2$ are independent copies of the limiting random measure $F(L)$ in Theorem~\ref{thm:2:m3}.

\section{Auxiliary results}\label{sec:aux}

	We will now present a few lemmas that we will use in our proofs. We will discuss properties of some random series as well as the asymptotic behaviour of the hitting
	times~\eqref{eq:1:FirstPassage}.

\subsection{Estimates for the related stochastic processes  $\{R_i\}_{i\in\Z}$ and $\{W_i\}_{i\in\Z}$. }

	We will frequently make use of the following notation: for integers $i\le j$,
	\begin{equation}\label{eq:1}
		\Pi_{i,j} = \prod_{k=i}^j \rho_k,\qquad R_{i,j} = \sum_{k=i}^{j} \xi_k\Pi_{i,k-1}, \qquad W_{i,j} = \sum_{k=i}^j \xi_k\Pi_{k,j}.
	\end{equation}
	We will also make use of the the limits
	\begin{equation}\label{eq:2}
		R_{i} = \lim_{j\to \infty} R_{i,j} =  \sum_{k=i}^{\infty} \xi_k\Pi_{i,k-1}, \qquad W_{j} = \lim_{i\to -\infty} W_{i,j} = \sum_{k=-\infty }^j \xi_k\Pi_{k,j}.
	\end{equation}
	Note that if $\E \log \rho < 0$ and $\E \log \xi < \infty$, both series are convergent as one can see by a~straightforward application of the law of large numbers and the Borel-Cantelli lemma (see  \cite[Theorem 2.1.3]{buraczewski:2016:power}).
	The random variables $R_i$'s and $W_j$'s have the same distribution and obey the recursive formulae
	\begin{equation*}
		R_i = \xi_i +\rho_iR_{i+1} \qquad \mbox{ and }\qquad  W_j = \rho_j\xi_j+\rho_jW_{j-1}.
	\end{equation*}
	We can therefore invoke the proof of \cite[Lemma 2.3.1]{buraczewski:2016:power} to infer the following result on the existence of moments of $R_i$'s and $W_j$'s.
	 In what follows we write $R$ (respectively $W$) for a generic element of $\{R_i\}_{i \in \ZZ}$ (respectively $\{W_j\}_{j \in \ZZ}$).

	\begin{lem}\label{lem:moments}
		Let $\alpha>0$. If $\E \rho^\alpha <1$, $\E \rho^\alpha\xi^\alpha <\infty$ and $\E \xi^\alpha <\infty$, then $\E R^\alpha$ and $\E W^\alpha$ are both finite.
	\end{lem}

\subsection{ Hitting times}
We describe now some properties of the sequence of stopping times $T = \{T_n\}_{n \in \NN}$ that allow us to better understand the process $X$ and indicate its ingredients which play an essential role in the proof of our
	main results.
	We will first analyse the hitting times $T$ along the marked sites $S$, that is
	$$
		T_{S_i} = \inf\{n:\; X_n = S_i\}.
	$$

	As it turns out, one can use $R_{i,j}$'s given in~\eqref{eq:1} to represent the exit probabilities from interval $(S_i, S_j)$. That is, for $i<k<j$ we have
	\begin{equation}\label{eq:3}
		\Po^{S_k}[T_{S_i} > T_{S_j}] = \frac{R_{i+1,k}}{R_{i+1,j}}, \qquad   \Po^{S_k}[T_{S_i} < T_{S_j}] = \Pi_{i+1,k}\frac{R_{k+1,j}}{R_{i+1,j}}.
	\end{equation}
	see the proof of~\cite[Theorem 2.1.2]{Zeitouni:2004:random}.
 Let
	$$
		\T_k = T_{S_k} - T_{S_{k-1}}
	$$
	be the time that the particle needs to hit $k$'th marked point $S_k$ after reaching $S_{k-1}$.
One uses $W_j$'s to describe the expected value of $\T_k$:
	\begin{equation}\label{eq:4}
	\Eo \T_k = 	\Eo^{S_{k-1}} T_{S_{k}} = \xi_{k}^2 + 2 \xi_{k}W_{k-1},
	\end{equation}
	see the proof of~\cite[Lemma 2.1.12]{Zeitouni:2004:random}.

Observe that the random variable $\T_k$  can be decomposed into a sum of two parts: the time the trajectory, after reaching $S_{k-1}$
	but before it hits $S_{k}$, spends to the left of $S_{k-1}$ and the time it spends to the right of $S_{k-1}$. For technical reasons that will become clear below, we divide the visits exactly at point $S_{k-1}$
	between these two sets depending on the direction from which the particle enters $S_{k-1}$. To be precise we define
	$$
		\T_k^{l} = \# \big\{ n\in (T_{S_{k-1}}, T_{S_k}]:\; X_n < S_{k-1} \ \mbox{ or } \ (X_{n-1}, X_n) = (S_{k-1}-1, S_{k-1} )   \big\},
	$$
	i.e. $\T_k^l$ is the sum of the time the particle spends in $(-\infty, S_{k-1}-1]$ and the number of steps
	from $S_{k-1}-1$ to $S_{k-1}$. Similarly we define
 	$$
		\T_k^{r} = \# \big\{ n\in (T_{S_{k-1}}, T_{S_k}]:\;  S_{k-1} < X_n \le S_k \ \mbox{ or } \ (X_{n-1}, X_n) = (S_{k-1}+1, S_{k-1} )   \big\}.
	$$
	Thus we can write
  	$$
		\T_k =  T_{S_{k}} - T_{S_{k-1}} = \T_k^{l} + \T_k^{r}.
	$$
  	Observe that given $\omega$, the random variables $\{\T_k\}_{k\in\N}$ are $\Po$ independent, however
  	for fixed $k$, $\T_k^{l}$ and $\T_k^{r}$ mutually depend on each other.
    	Summarizing, we obtain the following decomposition that will be used repeatedly:
  	$$
  		T_{S_k} = \sum_{j=1}^k \T_j = \sum_{j=1}^k \T^l_j + \sum_{j=1}^k \T^r_j =:   T^l_{S_k}+   T^r_{S_k}.
  	$$
	To proceed further we need to analyse  $\T^r_j$, $\T^l_j$ in details and describe their quenched expected value and quenched variance.
	Below we prove that after hitting any of the chosen sites $(S_k)_k$ the consecutive excursions to the left
  	are negligible. This entails that behaviour of $T_{S_k}$ is determined mainly by $T^r_{S_k}$.

\subsection{The sequence {$\{T_{S_n}^r\}$}.}

	Note that, under $\Po$, $\T_k^{r}$ equals in distribution to the time it takes a simple random walk on $[0, \xi_k]$ with a reflecting barrier placed in $0$ to reach $\xi_k$ for the first time
	when starting from $0$. This is the reason we include into $\T^r_{k}$ the visits at $S_{k-1}$, but only those from $S_{k-1}+1$.
	Indeed, let $(Y_n)_n$ be a simple random walk on $\Z$ independent of the environment $\omega$. Define
	\begin{equation}\label{eq:3:defUn}
		U_n = \inf\{ m : |Y_m| = n \},
	\end{equation}
	i.e. $U_n$ is the first time the reflected random walk hits $n$. Then  for every $k>0$, for fixed environment $\omega$,
	$\T_k^{r}\overset{d}{=} U_{\xi_k}$. In what follows we investigate how the asymptotic properties of $\xi_k$ affect those of $\T_k^{r}$. To do that, we will utilize the aforementioned
	equality in distribution and hence we first need to describe the asymptotic properties of $U_n$ as $n$ tends to infinity. The proof of the next lemma is omitted, since it follows from a
	standard application of Doob's
	optimal stopping theorem to martingales $Y_n^2-n$, $Y_n^4 - 6nY_n^2 + 3n^2 + 2n$, $ Y_n^6 - 15nY_n^4 + (45n^2+30n)Y_n^2 - (15n^3 + 30n^2 + 16n)$ and $\exp\{ \pm t Y_n\} {\rm cosh}(t)^{-n}$.

 	\begin{lem} \label{lem:m3}
		Let $U_n$, for $n \in \NN$ be given in~\eqref{eq:3:defUn}. We have
		\begin{equation*}
			 \E U_n = n^2, \quad   \E U_n^2  = 5n^4/3 - 2 n^2/3.
		\end{equation*}
 		Moreover, as $n \to \infty$,
 		\begin{equation*}
 			U_n/n^2 \Rightarrow 2\vartheta,
 		\end{equation*}
		for $\vartheta$ defined in \eqref{eq:mm1}. Furthermore the family of random variables $\{ n^{-4}U_n^2 \}_{n \in \NN}$ is uniformly integrable.
	\end{lem}

	The sequence $T^r_{S_n} = \sum_{k=1}^n \T_k^{r}$ is a sum of $\Po$ independent random variables $\{\T_k^r\}$.  Since, by Lemma~\ref{lem:m3},
\begin{equation}\label{eq:tvar}
\var_\omega \T_k^r = \frac 23 \xi_k^4 - \frac 23 \xi_k^2,
\end{equation}
the variance  $\var_\omega T_{S_n}^r$
	behaves asymptotically  as  $(2/3) \sum_{k=1}^n \xi_k^4$, thus obeys a stable limit theorem~\cite[Theorem 3.8.2]{durrett2019probability}. Moreover, we can use precise large deviation results for
	sums of i.i.d. regularly varying random variables~\cite[Theorem 9.1]{Denisov2008} to describe the deviations of $\var_\omega T_{S_n}^r$.  That is for any sequence $\{\alpha_n\}$ that
	tends to infinity,
	 $$
  		\P[\var_\omega T_{S_n}^r \ge \alpha_n a_n^4 ] \sim  (2/3)^{\beta/4} n \alpha_n^{-\beta/4}a_n^{-\beta} \ell(\alpha_n^{1/4}a_n).
  	$$
	We can now use Potter bounds~\cite[Theorem 1.5.6]{bingham:1987:regular} to control $\ell(\alpha_n^{1/4}a_n)$ with $\ell(a_n)$. This in turn yields a large deviation result
	asymptotic on the logarithmic scale.  We summarize this discussion in the following lemma.

	\begin{cor}\label{cor:m1}
  		The sequence $\{ \var_\omega T_{S_n}^r/ {a_n^4}\}_{n\in\N}$ converges in distribution (with respect to $\P$) to some stable random variable $Z$.
  		Moreover for any sequence $\{\alpha_n\}_{n \in \NN}$ that tends to infinity,
  		$$
  			\log\P[\var_\omega T_{S_n}^r \ge \alpha_n a_n^4 ] \sim - \beta\log(\alpha_n)/4.
  		$$
	\end{cor}

\subsection{The sequence $\{T_{S_n}^l\}$.}

   The structure of $\T_k^l$ is more involved. We may express it as a sum of independent copies of $F_k$, which denotes the length of a single excursion to the left from $S_k$, and thus obtain formulae for its quenched expectation and quenched variance.
   % in terms of independent copies of $F_k$, which denotes the length of a single excursion to the left from $S_{k}$.

	\begin{lem}\label{lem:m1}
		The following formulae hold
		\begin{equation}\label{eq:m1}
  		\begin{split}
			\Eo F_k &= 2(\xi_k + W_{k-1}),\\
			\var_\omega F_k &= 8\sum_{j<k} \Pi_{j+1,k-1}\left(\xi_{j+1}W_j^2 + \xi_{j+1}^2W_j + \frac{1}{3}(\xi_{j+1}^3 - \xi_{j+1})\right) \\
			&- 4W_{k-1}^2 + (-14\xi_k + 10) W_{k-1} - 4\xi_k^2 + 4\xi_k
		\end{split}
	\end{equation}
	and
	\begin{equation}\label{eq:m2}
		\begin{split}
			\Eo \T_k^l &= 2\xi_k W_{k-1}, \\%\rho_{k-1} (\xi_{k-1} + W_{k-2}), \\
			\var_\omega \T_k^l &= 8 \xi_k \sum_{j<k-1} \Pi_{j+1,k-1}\left(\xi_{j+1}W_j^2 + \xi_{j+1}^2W_j + \frac{1}{3}(\xi_{j+1}^3 - \xi_{j+1})\right) \\
			&+ \xi_k(1-\lambda_{k-1})\Big(\big(\xi_k(1-\lambda_{k-1}) + 2\lambda_{k-1}\big)(2W_{k-1} - \rho_{k-1})^2 \\
			& - 6(\xi_{k-1}-1)\rho_{k-1}W_{k-2} + \rho_{k-1}\Big).
		\end{split}
  		\end{equation}
	\end{lem}

	\begin{proof}
		Since the environment $\{(\rho_k,\xi_k)\}_{k\in\Z}$ is stationary, it is sufficient to calculate all the above formulae for $k=0$ or $k=1$. We first consider $\T^{l}_1$.
		Notice that the particle starting at 0 can return repeatedly to 0 from the right
  		or from the left. By the classical ruin problem, $\P^1_\omega(T_0 > T_{\xi_1}) = 1/\xi_1$. Thus the particle
  		starting at 1 hits the point 0 $M_1$ times before  it reaches $\xi_1$, where $M_1$ is geometrically distributed with parameter $1/\xi_1$ and mean $\xi_1-1$.
  		This is exactly the number of visits to $0$ counted by $\T^r_1$.
  		Between consecutive steps from 0 to 1, let's say between $m$th and $(m+1)$'th step, the particle spends some time
  		in $(-\infty, 0]$. In particular its visits at 0  from the left are exactly those included in $\T^l_1$.
  		Let us denote such an excursion by $G_0(m)$
		and denote by $G_0$ its generic copy. That is, $G_0$ ($G_k$, resp.) is the time the particle spends in $(-\infty,0]$ ($(-\infty,S_{k}]$, resp.) before visiting 1 ($S_k+1$, resp.). Then
		$G_0 \overset{d}{=} T_1 -1$ (or more generally $G_k \overset{d}{=} T_{S_k+1} - T_{S_k}-1$).

		The random variable $G_0$ consists of $N_m$ disjoint excursions in $(-\infty,-1]$, where $N_m$ is the number of jumps from $0$ to $-1$ before the next step to $1$.
		Since the  particle can jump to $-1$ with probability $(1-\lambda_0)$, $N_m$ has geometric distribution with mean $\rho_0$.
		Summarizing, $\T_1^l$ can be decomposed as
		\begin{equation}\label{eq:w1}
			\T_1^{l} = \sum_{m=0}^{ M_1} G_0(m) = \sum_{m=0}^{ M_1} \sum_{j=1}^{N_m} F_0({j,m}),
		\end{equation}
		where $F_0({j,m})$ measures the length of a single left excursion from 0.
		Observe that both $N_m$'s and $F_0({j,m})$'s are i.i.d. under $\Po$. Moreover, the first sum includes $m= 0$, because the process starts at $0$.

		Recall that if $S_N = \sum_{k=1}^N X_i$ for some random variable $N$ and an i.i.d. sequence $\{X_n\}$ independent of $N$, then
		\begin{equation}\label{eq:t11}
			\var S_N = \E N \cdot \var X + \var N \cdot (\E X)^2.
		\end{equation}
		The above formula together with  \eqref{eq:w1} easily entails
  		\begin{equation}\label{eq:4:1}
  			\begin{split}
  			\Eo \T_1^{l} &= \xi_1\rho_0 \Eo F_0,\\
			\var_\omega \T_1^l &= \xi_1\rho_0 \var_\omega F_0
			+   \left(\xi_1^2 \rho_0^2 + \xi_1\rho_0 \right)(\Eo F_0)^2.
  			\end{split}
		\end{equation}
		Since $F_0$ is the time of a single excursion from $0$ that begins with a step left, using the solution to the classical ruin problem in combination with formula~\eqref{eq:4} we get
  		$$
			\Eo F_0 = 1 +  \Eo^{-1}  T_0 = 1 + (\xi_0 - 1) + \frac{1}{\xi_0} \Eo^{S_{-1}} T_{S_0} = 2(\xi_0 + W_{-1}),
		$$
		A formula for quenched variance of crossing times for arbitrary neighbourhood was given in~\cite[Lemma 3]{goldsheid2007simple} and yields \eqref{eq:m1}. Inserting these formulae
		to \eqref{eq:4:1}, using the fact that $\rho_0W_{-1} = W_0 - \xi_0\rho_0$, and finally simplifying the expression leads to~\eqref{eq:m2}.
	\end{proof}

	\begin{lem}\label{lem:var}
		For every $\varepsilon > 0$ and  $\theta \ge 0$,
		$$
			\P\big[ \var_\omega  T_{S_n}^l \ge \varepsilon n^{\theta} a_n^4 \big] \le o(1)/n^{\theta \gamma}, \quad n\to\infty,
		$$
		where $\gamma$ is a parameter satisfying \eqref{eq:2:assumption2}. In particular,
  		$$
  			\frac 1{a_n^4} \var_\omega T_{S_n}^l
  			\overset{\P}{\to} 0.
  		$$
	\end{lem}

	\begin{proof}
  		To prove the lemma  one needs to deal with the formula for the variance \eqref{eq:m2}.
  To avoid long and tedious arguments  we will explain how to estimate  two of the terms, i.e. we will
  prove
  \begin{equation}\label{eq:s1}
   \P\bigg[   \sum_{k=1}^n \xi_k \cdot \sum_{j <  k-1} \Pi_{j+1,k-1}\xi_{j+1}  W_{j}^2  \ge \varepsilon n^{\theta} a_n^4 \bigg] \le o(1)/n^{\theta \gamma}, \quad n\to\infty
  \end{equation}
  and
  \begin{equation}\label{eq:s2}
    \P\bigg[ \sum_{k=1}^n \xi_k \cdot \sum_{j < k-1} \Pi_{j+1,k-1} \xi_{j+1}^3  \ge \varepsilon n^{\theta} a_n^4 \bigg] \le o(1)/n^{\theta \gamma}, \quad n\to\infty.
  \end{equation}
  All the remaining terms can be treated using exactly the same arguments.

Recall $\gamma \in (\beta/4,1\wedge \beta)$ and $\E \rho^{2\gamma}<1$.
The Markov inequality and  independence of
$\xi_k$, $\Pi_{j+2,k-1}$, $\rho_{j+1}\xi_{j+1}$ and $W_{j}$  yield
\begin{equation*}
  \begin{split}
     \P\bigg[     \sum_{k=1}^n \xi_k & \cdot \sum_{j < k-1} \Pi_{j+1,k-1}\xi_{j+1} W_{j}^2  \ge \varepsilon n^{\theta} {a_n^4}  \bigg]
     \le \frac{1}{\varepsilon^\gamma n^{\theta \gamma } a_n^{4\gamma}} \E
     \bigg[     \sum_{k=1}^n \xi_k \cdot \sum_{j < k-1} \Pi_{j+1,k-1}\xi_{j+1} W_{j}^2  \bigg]^\gamma\\
&     \le \frac{1}{\varepsilon^\gamma n^{\theta \gamma } a_n^{4\gamma}}
\sum_{k=1}^n \E \xi^{\gamma}_k \cdot \sum_{j <  k-1} \E \Pi_{j+2,k-1}^\gamma \E [\rho_{j+1}^\gamma \xi_{j+1}^\gamma] \E W_{j}^{2\gamma}
\le \frac{Cn}{\varepsilon^\gamma n^{\theta \gamma } a_n^{4\gamma}} = \frac{o(1)}{n^{\theta \gamma }},
  \end{split}
\end{equation*}
where the last inequality follows from our hypotheses \eqref{eq:2:assumption2} and Lemma \ref{lem:moments}.
This proves \eqref{eq:s1}.
We proceed similarly with the second formula \eqref{eq:s2}:
\begin{equation*}
  \begin{split}
     \P\bigg[     \sum_{k=1}^n \xi_k & \cdot \sum_{j <  k-1} \Pi_{j+1,k-1}   \xi_j^3  \ge \varepsilon n^{\theta }{a_n^4}  \bigg]
     \le \frac{1}{\varepsilon^\gamma a_n^{4\gamma}} \E
     \bigg[     \sum_{k=1}^n \xi_k \cdot \sum_{j <  k-1} \Pi_{j+1,k-1}  \xi_j^3  \bigg]^\gamma\\
&     \le \frac{1}{\varepsilon^\gamma n^{\theta \gamma } a_n^{4\gamma}}
\sum_{k=1}^n \E \xi^{\gamma}_k \cdot \sum_{j < k-1} \E \Pi_{j+2,k-1}^\gamma \E [\rho_{j+1}^\gamma \xi_{j+1}^{3 \gamma}]
\le \frac{Cn}{\varepsilon^\gamma n^{\theta \gamma } a_n^{4\gamma}}  = \frac{o(1)}{n^{\theta \gamma }}.
  \end{split}
\end{equation*}

Invoking the first part of the lemma with $\theta = 0$ we conclude
convergence of  $\var_\omega  T_{S_n}^l / a_n^4$ to 0 in probability.
\end{proof}

%Below we will use the dominant role of $T_n^r$ and will need to express it in terms of the variance
%\begin{lem}\label{lem:variance}
%	For any positive constant $C > 0$ we have
%	$$
%		-\var_\omega \T^r_n/C + (1- C) \var_\omega  \T^l_n
%		\le
%		\var_\omega \T_n - \var_\omega \T^r_n \le \var_\omega \T^r_n/C + (1+C) \var_\omega  \T^l_n.
%	$$
%\end{lem}
%\begin{proof}
%	Given $n$ we write
%	$$
%		\var_\omega \T_n = \var_\omega (\T^r_n + \T^l_n)= \var_\omega \T^r_n + \var_\omega  \T^l_n
%		+ 2{\rm Cov}_\omega(\T^r_n, \T^l_n).
%	$$
%	Applying the Schwartz inequality and the well-know inequality $2ab\le a^2/C + C b^2$ we have
%	$$
%		2{\rm Cov}_\omega(\T^r_n, \T^l_n) \le 2 \sqrt{\var_\omega \T^r_n \cdot \var_\omega  \T^l_n}
%		\le \var_\omega \T^r_n/C + C \var_\omega  \T^l_n
%	$$
%	and we easily conclude the Lemma.
%\end{proof}

\section{Absence of a strong limit}\label{sec:absence}

	Our aim now is to prove Theorem \ref{mthm0} saying that the 'strong limit in distribution' does not
 exist.  For most of the proof we will consider the standard normalization that is $C_n(\omega) = \Eo T_n$, $A_n(\omega) = \sqrt{\var_\omega T_n}$ and  study the normalized sequence
 	\begin{equation}\label{eq:t7}
		\widetilde T_n = \frac{T_n - \E_\omega T_n}{\sqrt{\var_\omega T_n}}.
 	\end{equation}
We will prove that for $\P$-a.e. $\omega$ there exists its subsequence $\{\widetilde T_{n_k(\omega)}\}$ convergent to $2\vartheta - 1$, however on the other
hand as we will see this cannot be the limit of the whole sequence. Finally we will show that there is no other normalization
leading to a nontrivial limit.
%	there is no non-zero random variable $Y_\omega$ such that

Absence of the strong quenched limit follows essentially from the fact that for $\P$-a.e. $\omega$ one can  find
	an infinite subsequence $\{\xi_{n_l}\}_{l\in\N}$ such that the values of $\xi_{n_l+1}$ are exceptionally  large,
	hence $T_{S_{n_l+1}} - T_{S_{n_l}}$, the time the walk needs to move from $S_{n_l}$ to $S_{n_l+1}$, is
	either much bigger or comparable  with $T_{S_{n_l}}$.

	First we need to construct a favourable environment of probability one.  For this purpose
	we consider two increasing sequences  $\{p_n\}$, $\{q_n\}$ diverging to $+\infty$ such that
	\begin{equation}\label{eq:m3}
		%q_n - 2 p_n \to \infty, \qquad
%q_n/p_n \to \infty \qquad \mbox{and} \qquad
2 p_n < q_n < p_{n+1}/2,  \qquad p_n/q_n \to 0 \qquad \mbox{and} \qquad
		\frac{a_{q_n}}{a_{2p_n}} \ge n^{\theta}
	\end{equation}
	for  $\theta > \max\{1/(4\gamma), 1/\beta\}$.  	Notice that one may take e.g. $p_n = 2^{2^n}$, $q_n=p_{n+1}/4$.

	We will need to prove that behaviour of the process in the interval $[S_{2p_n}, S_{2q_n}]$ is determined
	by its position after time $T_{2p_n}$ and that its previous values
	up to time $2p_n$ are negligible when looking
	at time $q_n$ and scale $a_{q_n}$.
% For technical reasons it will be easier to express
%this property in terms of the variance,    so we  need to find parameters such that
%	\begin{equation}\label{eq:m4}
%  		\P\big[\forall \varepsilon >0\  \var_\omega T_{S_{2p_n}}
 % \ge a^4_{q_n} \varepsilon \ \mbox{i.o.} \big] = 0.
%	\end{equation}
%	Therefore we assume additionally that
%	\begin{equation}\label{eq:m5}
%		\frac{a_{q_n}}{a_{2p_n}} \ge n^{\theta}
%	\end{equation}
%	for  $\theta > \max\{1/(4\gamma), 1/\beta\}$. In view of the Borel - Cantelli Lemma, this hypotheses combined with Corollary \ref{cor:m1} and Lemma \ref{lem:var} yields \eqref{eq:m4}.
	The trajectory of the random walk  $X$ %$\{X_k\}$
cannot be divided into independent pieces with respect to $P$, because the process can have large excursions to the left
 and the environment is not homogeneous.
	To remedy that we will censor the left excursions of $X$ that become too large.
	We consider
	a new process, say $\overline{X}=\{\overline X_k\}_{k \in \NN}$. This process essentially behaves as the previous one and evolves in the same environment, with a small difference. Namely
	after $\overline{X}$ reaches $S_{q_n}$ and before it reaches $S_{2q_n}$ we put a barrier at point $S_{p_n}$,  i.e. the process cannot come back below
	$S_{p_n}$.  However this barrier is removed when $\overline X$ hits  ${S_{2q_n}}$. Of course we can couple both processes on the same probability space removing all left excursions from $S_{p_n}$
	after hitting $S_{q_n}$ and before reaching $S_{2q_n}$.

	For any $k$, we  define the random variables $\overline T_k$, $\overline \T_k$, $\overline \T^r_k$, $\overline \T^l_k$
in an obvious way, e.g.
$$
\overline T_k = \inf\{j:\; \overline X_j = k  \}, \quad \overline \T_k =  \overline T_{ S_k} - \overline T_{ S_{k-1}}.
$$
% counting only
%	the positions of $X$ that are above $S_{p_n}$. For example %for $k\in ( q_n, 2 q_n]$
%we set
%	$$
%		\overline \T_k = \# \big\{j \in ( T_{S_{k-1}}, T_{S_k}] \mbox{ such that }  X_j \ge S_{p_n}  \big\}.
%	$$
 	Then $\overline \T^r_k = \T^r_k$ for all $k'$s and $\overline \T^l_k = \T^l_k$ for
	$k \notin \bigcup_n (q_n, 2q_n]$.
	Notice that $\T_k - \overline \T_k$ is the time that the process $X$ spends below $S_{p_k}$ after hitting $S_{k-1}$ and before reaching $S_k$.
The next lemma ensures that asymptotic properties of the processes $X$ and $\overline X$ are comparable.

\begin{lem}\label{lem:m5}
For any $\varepsilon\in (0,1)$ and $\P$-a.e. $\omega$ there is $N = N(\omega)$ such that
\begin{equation}\label{eq:t21}
  \sum_{q_n<k\le 2q_n} \E_\omega(\T_k - \overline \T_k) < \varepsilon^n
\end{equation}
and
\begin{equation}\label{eq:t22}
  \sum_{q_n<k\le 2q_n} \var_\omega(\T_k - \overline \T_k) < \varepsilon^n
\end{equation}
for $n>N$.

		Moreover
		$$
			\T_{n} = \overline \T_n \ \mbox{a.s. for large (random) n.}
		$$
\end{lem}
	\begin{proof}
Fix $k \in ({q_n}, {2q_n}]$. To describe the quenched mean and the quenched variance of $\T_k - \overline \T_k = \T^l_k - \overline \T^l_k$
we need to calculate the time the trajectory $X$, after it hits $S_{k-1}$, but before reaching $S_k$, spends below $S_{p_n}$. For this
purpose we proceed as in the proof of Lemma \ref{lem:m1}, that is we decompose

\begin{equation}\label{eq:t23a}
	\T_k - \overline \T_k = \sum_{m=1}^{M_k} \sum_{j=0}^{N_m} F_{p_n} (j,m),
\end{equation}
where $M_k$ denotes the number of times the walk visits $S_{p_n}$ from the right in the time interval $(T_{S_{k-1}}, T_{S_k})$, $N_m$ is the number of consecutive left excursions from $S_{p_n}$ after hitting it from the right, and $F_{p_n}(j,m)$ is the length of the corresponding excursion. Note that $N_m$ is geometrically distributed with mean $\rho_{p_n}$ and variance $\rho_{p_n}(1 + \rho_{p_n})$. Thus, by formulae \eqref{eq:t11} and \eqref{eq:m1},
\begin{equation}\label{eq:sumN}
	\begin{split}
		\Eo \bigg[ \sum_{j=0}^{N_m} F_{p_n}(j,m) \bigg] &= \rho_{p_n} \Eo F_{p_n} = 2W_{p_n}\\
		\var_\omega \bigg[ \sum_{j=0}^{N_m} F_{p_n}(j,m) \bigg] &= \rho_{p_n} \var_\omega F_{p_n} + \rho_{p_n} (1+\rho_{p_n}) \left( \Eo F_{p_n} \right)^2.
	\end{split}
\end{equation}
Next, observe that for any $m > 0$, $\Po \left[M_k = m\right] = r s^{m-1} (1-s)
$, where
$$
r = \Po^{S_{k-1}} \left[ T_{S_{p_n}} < T_{S_k} \right]
$$
and, invoking once again the gambler's ruin problem,
$$
		s = \Po^{S_{p_n} + 1} \left[ T_{S_{p_n}} < T_{S_k} \right]
		= 1 - \frac{1}{\xi_{p_n+1}} \Po^{S_{p_n + 1}}\left[ T_{S_{p_n}} > T_{S_k}\right].
$$
We may easily calculate the mean and variance of $M_k$ and use the formulae \eqref{eq:3} to express them in terms of the environment. We get, after simplifying,
\begin{equation} \label{eq:M_k}
\begin{split}
	\Eo M_k &= \frac{r}{1-s} = \xi_k \Pi_{p_n + 1, k-1}, \\
	\var_\omega M_k &= \frac{r(1+s-r)}{(1-s)^2} = \xi_k\Pi_{p_n+1,k-1}\left(2R_{p_n+1,k-1} + \xi_k\Pi_{p_n + 1, k-1} - 1 \right).
\end{split}
\end{equation}
Therefore, by \eqref{eq:1}, \eqref{eq:sumN} and \eqref{eq:M_k},
\begin{equation}\label{eq:t56a}
	\begin{split}
		\Eo \left[ \T_k - \overline \T_k \right] & = 2 \xi_k \Pi_{p_n+1, k-1} W_{p_n}, \\
		\var_\omega \left[ \T_k - \overline \T_k \right] & =
		\xi_k \Pi_{p_n, k-1} \var_\omega F_{p_n} \\
		%& + \xi_k \Pi_{p_n, k-1} \left(\Eo F_{p_n}\right)^2 \left( 1
		%+  \rho_{p_n} ( 2R_{p_n+1,k-1} + \xi_k\Pi_{p_n + 1,k-1})\right) = \\
		%&\xi_k \Pi_{p_n, k-1} \var_\omega F_{p_n} \\
		& + \xi_k \Pi_{p_n, k-1} \left(\Eo F_{p_n}\right)^2 \left( 1
		+  2\sum_{j=p_n+1}^{k-1} \xi_j \Pi_{p_n,j-1} + \xi_k\Pi_{p_n,k-1}\right) %\\
%		& = 8\xi_k \Pi_{p_n + 1,k-1} \sum_{j\leq p_n} \Pi_{j,p_n-1} \left(\xi_j W_{j-1}^2 + \xi_j^2 W_{j-1} + \frac{1}{3}(\xi_j^3 - \xi_j)\right) \\
%		& + \xi_k \Pi_{p_n + 1,k-1}\left( -6\xi_{p_n}W_{p_n-1} + 10W_{p_n-1} + 4\xi_{p_n}\right) \\
%		& + 2\left( \xi_k \sum_{p_n + 1}^k \xi_j \Pi_{p_n+1,j-1}^2 \Pi_{j,k-1} - \xi_k^2 \Pi_{p_n+1,k-1}^2 \right) W_{p_n}^2
.
	\end{split}
\end{equation}

Now, we are ready to prove \eqref{eq:t21}. We have
%\begin{align*}
$$  \P\bigg[ \sum_{q_n<k\le 2q_n} \E_\omega(\T_k - \overline \T_k) \ge \varepsilon^n \bigg]
\le \varepsilon^{-n} \sum_{q_n < k \le 2q_n} \E\big[ 2\xi_k \Pi_{p_n+1, k-1}W_{p_n}  \big]^\gamma
\le C \varepsilon^{-n} (\E \rho^\gamma)^{q_n-p_n},$$
where $\gamma \in (0,1)$ is a small constant such that $\E \rho^\gamma < 1$, $\E \xi^\gamma <\infty$ and $\E W^\gamma <\infty$ (see
\eqref{eq:2:assumption2} and Lemma \ref{lem:moments}).
Then, by the Borel-Cantelli lemma
$$
\P\bigg[ \sum_{q_n<k\le 2q_n} \E_\omega(\T_k - \overline \T_k) \ge \varepsilon^n \ \ \mbox{ i.o. } \bigg] = 0,
$$
which gives \eqref{eq:t21}. Formula \eqref{eq:t22} can be proved applying essentially the same argument. One needs
to compute, combining \eqref{eq:m1} with \eqref{eq:t56a}, the precise expression for the variance and
next, arguing as above and considering each of the summands separately (as in the proof of Lemma \ref{lem:m1}),
%and  using long products of $\rho$'s
one can deduce \eqref{eq:t22}. We skip the details.

Finally we write
		\begin{equation*}
			\PP[\T_{k} \neq \overline \T_{k}] = \E \left[ \Po [\T_{k} - \overline \T_{k} \geq 1]  \right] \leq \E \left[ \Eo (\T_{k} - \overline \T_{k})    \right]
			\leq C   \big(\E \rho^\gamma\big)^{ k - p_n}
		\end{equation*}
		to infer our final claim by yet another appeal to the Borel-Cantelli lemma.
	\end{proof}

%	Similarly {\red as in the proof of Lemma  \ref{lem:m1}}, we also define $\overline G_k$ to be the length of the left excursion of {\red $\overline X$ } from $S_k$ before hitting $S_{k}+1$.
%	Of course $\E_\omega \overline G_k \le \E_\omega G_k$ and $\var_\omega \overline G_k \le \E_\omega G_k^2$.
	The advantage of introducing a new process $ \overline X$ is that it behaves similarly to
 $X$  and from the point of view of limit theorems this change is indistinguishable. However,   here one can indicate  independent pieces:
	$\{\overline X_k\}_{k\in ( \overline T_{q_n}, \overline T_{2q_n}] }$ are $\P$-independent.

	Now we are ready to describe the required properties of the environment. The sets below depend on several parameters.
	Given $d < D$ and  $b< B$, and $\varepsilon > 0$ let
	
%	{\color{ala} \tt Zmienilam $a_{k+1}^{1/2}$ na $a_{k+1}$ przy $G$, bo to wystarczy, i $\xi/a$ na $\xi^4/a^4$, bo inaczej wszedzie dalej trzeba by pozmieniac $b,B$ na $b^4, %B^4$. }
  	\begin{multline*}
		U_n(d,D,b,B,\varepsilon)  = \bigg\{\mbox{there exists } k\in (q_n, 2q_n] \mbox{ such that }
 			%\frac{ S_k - S_{2p_n}}{a_{k+1}} \in (d,D), \\
			\frac{ \var_\omega(\overline T^r_{ S_k} - \overline T^r_{ S_{2p_n}})}{a^4_{k+1}} \in (d,D), \\
			\frac{  \var_{\omega} (\overline T^l_{ S_k} - \overline T^l_{ S_{2p_n}})}{a^4_{k+1}} \le \varepsilon,\quad
			\frac{ \var_{\omega} \overline G_{k} + (\E_{\omega} \overline G_{k})^2  }{ a_{k+1}}  \le \varepsilon, \quad
 			%\quad   M_{k+1} \le \xi_{k+1},
			\quad  \frac{\xi^4_{k+1}}{a^4_{k+1}}\in (b,B) \bigg\},
      \end{multline*}
      where  $\overline G_k$ is the length of the left excursion of  $\overline X$  from $S_k$ before hitting $S_{k}+1$.
	Of course $\E_\omega \overline G_k \le \E_\omega G_k$ and $\var_\omega \overline G_k \le \E_\omega G_k^2$.
      %where $M_{k+1}$ is the number of returns to $S_k$ from the right before hitting $S_{k+1}$. Hence $M_{k+1}$
       %is a random variable which conditioning on $\xi_{k+1}$ has the geometric law with parameter $1/\xi_{k+1}$
 	%and is independent of all the remaining random variables.
	We want to consider environments which belong to infinitely many sets $U_n$. However, given $\omega$, we want to have some freedom of choosing all the parameters.
	The lemma below justifies that the measure of these environments is one.

	\begin{lem}\label{lem:dom}
		Assume that conditions \eqref{eq:2:assumption} and \eqref{eq:2:assumption2} are satisfied. Then the event
		$$
			\mathcal{U}= \bigcap\left\{  \limsup_{n} U_n (d,D,b,B,\varepsilon)  \: : \: d,D,b,B \in {\mathbb Q}^+, \: d< D, b< B, \: b > 3\cdot 2^{4/\beta} D, \: \varepsilon > 0 \right\}
		$$
		has probability one.
		
%		{\color{ala} \tt Czy nie wystarczy $b> 3\cdot 2^{4/\beta}D$?  (To sie pojawia przy uzasadnieniu pustego przekroju pod koniec dowodu.)}
	\end{lem}

	\begin{proof}
		Since in the above formula the intersections are essentially over a countable set of parameters (one can obviously restrict to the rational parameter $\varepsilon$),
		it is sufficient to prove that for fixed parameters  $d < D$, $b< B$ such that $b > 3\cdot 2^{5/\beta-1} D$ and $\varepsilon > 0$,
		$$
			\P\big[ \limsup_n U_n \big] = 1,
		$$
		for $U_n = U_n (d,D,b,B,\varepsilon)$. Observe that the events $\{U_n\}_{n\in\N}$ are independent, because $U_n$ depends only
		on $\{\omega_j\}_{j\in [p_n, 2q_n]}$ and thanks to \eqref{eq:m3} the sets  $\{[p_n, 2q_n]\}_{n\in \N}$ are pairwise disjoint. Thus, invoking the Borel-Cantelli Lemma,
		it is sufficient to prove that there is $\delta_0>0$ such that for large indices $n$,
\begin{equation}\label{eq:t1}
\P[U_n] > \delta_0.
\end{equation}

		We need to estimate probabilities of all the events which appear in the definition of $U_n$. Denote
		\begin{equation}\label{eq:m10}
 		\begin{split}
    			V_k^1 & = \left\{  \var_\omega(\overline T^r_{S_k} - \overline T^r_{S_{2p_n}})/{a^4_{k+1}} \in (d,D) \right\},
     \quad
    			V_k^2 = \left\{ \var_{\omega} (\overline T^l_{S_k} - \overline T^l_{S_{2p_n}}) / a^4_{k+1} \le \varepsilon \right\}, \\
			V_k^3  &= \left\{     \big(\var_{\omega} \overline G_{k} +   (\E_{\omega} \overline G_{k})^2\big)/ a_{k+1} \le \varepsilon \right\},
   		\quad    V^4_{k+1}  = \left\{\xi^4_{k+1}/a^4_{k+1}\in (b,B) \right\}.
  		\end{split}
  		\end{equation}
 		To estimate the probability of $V_k^1$ notice that thanks to \eqref{eq:m3} we have $a_{k-2p_n}/a_{k+1} \to 1$
 %{\color{ala} \tt czy tu nie potrzebujemy zalozenia na $p_n/q_n$, ktore zniklo?} {\color{p} \tt potrzebujemy $p_n/q_n \to 0$}
 for any $k\in (q_n, 2q_n]$,
 therefore
		\begin{equation*}
			\P[V_k^1] = \P \left[\frac{ 2/3\sum_{2p_n < j \le k } (\xi^4_k - \xi_k^2)}{a_{k - 2p_n}^4} \cdot \frac{a_{k - 2p_n}^4}{a^4_{k+1}} \in (d,D) \right] \xrightarrow{n\to\infty} \delta \in (0,1),
		\end{equation*}
		by Lemma \ref{lem:m3}. Recall that the second claim of Lemma \ref{lem:var} means exactly that $\P[V_k^2] \to 1$.
		Thirdly, recalling that $\var_\omega \overline G_k +  (\E_{\omega} \overline G_{k})^2  \le \E_\omega G_k^2 +  (\E_{\omega} G_{k})^2$, which is a stationary sequence, we obtain
		$(\var_\omega \overline G_k + (\E_{\omega} \overline G_{k})^2 )/ a_k\overset{\P}{\to} 0$, i.e. $\P[V_k^3] \to 1$.
%		Next, to estimate the probability of $\{M_{k+1} \le \xi_{k+1}\}$, observe that if  a random variable $Z$  is geometrically distributed with some parameter $p$, then
%		$$
%			\P[ Z \le 1/p] = \sum_{k\le 1/p} p(1-p)^k = 1 - (1-p)^{\lfloor 1/p \rfloor} \to 1 - 1/e > 1/2
%		$$
%		as $p\to 0$.  Therefore, on the set where $\xi_{k+1}$ is  sufficiently large, $\P[M_{k+1} \le \xi_{k+1} | \xi_{k+1}] > 1/2$.
		Finally, observe that thanks to our choice of parameters the events defining $U_n$ are disjoint for different values of $k$'s.
		Indeed, if $k,j\in (q_n, 2q_n]$, $k < j$ and $\xi_{k+1}^4 \ge b a_{k+1}^4$, then $V^4_{k+1}\cap V_j^1 = \emptyset$,
because by Lemma \ref{lem:m3}, for large $n$, we have
  $$
\frac{\var_{\omega} (\overline T^r_{S_j} - \overline T^r_{S_{2p_n}}) }{a^4_{j}} \sim   \frac{\frac 23 \sum_{i=2 p_n}^j \xi_i^4}{a_j^4} \ge
\frac 23\;  \frac{ \xi_{k+1}^4}{a_{2q_n}^4} \ge
\frac 23\;  \frac{ b a_{q_n}^4}{a_{2q_n}^4} \ge  \frac{b}{ 3\cdot 2^{ 4/\beta}} > D.
  $$
%
%because on $V^4_{k+1}$ for sufficiently large $n$,
 % 		$$
%			\frac{ S_j - S_{2p_n} }{a_{j}}  \ge \frac{ \xi_{k+1}}{a_{2q_n}} \ge \frac{ b a_{q_n}}{a_{2q_n}} \ge  \frac{b}{ 2^{1/\beta}} > D.
%		$$
  		Summarizing, for some $\delta'>0$ and large $n$
		\begin{align*}
			\P[U_n]  &=  \P\bigg[ \bigcup_{q_n < k \le 2q_n} V_k^1 \cap V_k^2\cap V_k^3  \cap  V^4_{k+1}  \bigg]
			 = \sum_{q_n < k \le 2 q_n }\P\left[ V_k^1 \cap V_k^2 \cap V_k^3  \right] \P \left[     V^4_{k+1}  \right] \\
		%	&\ge \sum_{q_n < k \le 2 q_n } \left( \P\left[ V_k^1  \right] - \P\left[ (V_k^2)^c \right] \right)
		%		\cdot \E\left[ {\1}_{\{ \xi_{k+1} / a_{k+1} \in (b,B)  \}} \E \left[ {\1}_{\{M_{k+1} \le \xi_{k+1}\}}   \big| \xi_{k+1}  \right]   \right] \\
			&\ge \sum_{q_n < k \le 2 q_n } \frac{\delta}{2}  \cdot \P\left[  \xi_{k+1} / a_{k+1} \in (b,B) \right]
			\ge \frac{\delta \delta'}{2}  \sum_{q_n < k \le 2 q_n } \frac 1k \sim  \frac{\delta \delta'\log 2}{2}.
		\end{align*}
		In conclusion, the probabilities of $U_n$ are bounded from below, which  entails \eqref{eq:t1} and
completes the proof.
	\end{proof}

	\begin{proof}[Proof of Theorem \ref{mthm0}]
In view of  Lemma \ref{lem:var} and our hypothesis \eqref{eq:m3}, the Borel-Cantelli lemma yields
$$
 		\P\big[\forall \varepsilon >0\  \var_\omega T_{S_{2p_n}} \ge a_{q_n}^4 \varepsilon \ \ \mbox{i.o.}\big] = 0.
$$ Therefore, invoking  Lemma~\ref{lem:dom}  the set
		$$
			\mathcal{U} \cap  \big\{\forall \varepsilon >0\   \var_\omega T_{S_{2p_n}} \ge a^4_{q_n} \varepsilon \ \mbox{i.o.} \big\}^c
		$$
		has probability $1$. From now we fix $\omega$ from the event above which also satisfies the claim of Lemma~\ref{lem:m5}.
		Assume that, given~$\omega$,
		\begin{equation}\label{eq:m6}
  		\widetilde T_n=	\frac{T_n - \E_\omega T_n}{\sqrt{\var_\omega T_n}} \Rightarrow Y_\omega \qquad n\to\infty,
  		\end{equation}
		for some random variable $Y_\omega$.
		We fix  parameters  $d < D$, $b< B$ such that $b >  3\cdot 2^{5/\beta-1} D$ and $\varepsilon > 0$.
%		Next we choose $\varepsilon > 0$.%, whose precise value will be specified below.

		Take two sequences $\{n_m\}_{m\in\N}$ and $k_{m} \in (q_{n_m},q_{2n_m}]$ such that
		$$
			\omega \in V^1_{k_{m}}\cap V^2_{k_{m}}\cap V^3_{k_{m}} \cap V^4_{k_{m}+1},
		$$
		where all the sets were defined in \eqref{eq:m10}. %By the merit of~\eqref{eq:m4}
We can additionally assume (removing a finite number of elements of the sequence if needed),
		that for all indices $m$
		\begin{equation}\label{eq:m11}
\var_\omega T_{S_{2p_{n_m}}} <  a^4_{k_{m}} \varepsilon.
		\end{equation}
%and in view of Lemma \ref{lem:var} (passing, by the Riesz theorem, to a subsequence if needed) that
%\begin{equation}\label{eq:t5}
%  \var_\omega \T^l_{k_m+1} = o(a_{k_m+1}^4), \quad k_m\to \infty.
%\end{equation}

	 	Lemma \ref{lem:m5} says that, given $\omega$, the difference $(T_n - \E_\omega T_n ) - (\overline T_n - \E_\omega \overline T_n) $
remains a.s. bounded and $\var_\omega T_n / \var_\omega \overline T_n $ converges to 1, hence
 \eqref{eq:m6} yields
		\begin{equation}\label{eq:mm3}
 	 		\frac{\overline T_{n} - \E_\omega \overline T_n}{\sqrt{\var_\omega \overline T_n}} \Rightarrow Y_{\omega} \qquad n\to \infty.
	  	\end{equation}
		Consider the following decomposition,
		\begin{equation}\label{eq:m7}
			\frac{\overline T_{S_{k_m+1}} - \E_\omega \overline T_{S_{k_m + 1}}} {\sqrt{\var_\omega \overline T_{S_{k_m+1}}}}
				= v_{m} \cdot V_m +w_m \cdot W_m + Z_m,
		\end{equation}
		where
		\begin{align*}
			&V_m = \frac{\overline T_{S_{k_{m}}} - \E_\omega \overline T_{S_{k_{m}}}}{\sqrt{\var_\omega \overline T_{S_{k_m}}}},
			&v_{m} = \frac{\sqrt{\var_\omega \overline T_{S_{k_m}}} }{\sqrt{\var_\omega \overline T_{S_{k_m+1}}}},\\
			& W_m = \frac{\overline \T^r_{k_m+1} - \E_\omega \overline \T^r_{k_m+1}}{\xi_{k_m+1}^2},	
			&w_{m} = \frac{   \xi_{k_m+1}^2}{\sqrt{\var_\omega \overline T_{S_{k_m+1}}}},\\
			&Z_m = \frac{\overline \T^l_{k_m+1} - \E_\omega \overline \T^l_{k_m+1}}{\sqrt{\var_\omega  \overline T_{S_{k_m+1}}}}.
		\end{align*}
		Random variables $V_m$ and $(W_m,Z_m)$ are $\P_\omega$-independent. By \eqref{eq:mm3}, $V_m$ converges in distribution to $Y_\omega$, whereas $W_m$, by Lemma \ref{lem:m3}, converges to
		$2\vartheta - 1$. Therefore we need to understand the behaviour of both deterministic (given $\omega$) sequences $\{v_m\}_{m\in \N}$, $\{w_m\}_{m\in \N}$ and
of the sequence of random variables $\{Z_m\}$. For this purpose we need to understand behaviour of the variances which appear in the above formulae.
Note first that on the considered event, recalling \eqref{eq:t11}, we have
		\begin{equation}\label{eq:t2}
			\var_\omega \overline \T^l_{k_m+1} \le  \xi_{k_m+1}\var_\omega \overline G_{k_m} + \xi^2_{k_m+1} (\E_\omega \overline G_{k_m})^2 \le\varepsilon  a_{k_m+1}^3 { B^{1/2}. }% (B+B^2).
		\end{equation}
Applying the Schwartz inequality and the well-known inequality $2ab\le a^2/C + C b^2$, for any $n$ and arbitrary large constant $C$, we can easily prove
$$
-\var_\omega \overline \T^r_n/C + (1- C) \var_\omega  \overline \T^l_n
\le
\var_\omega \overline \T_n - \var_\omega \overline \T^r_n \le \var_\omega \overline \T^r_n/C + (1+C) \var_\omega \overline \T^l_n.
$$
Combining the above inequalities with \eqref{eq:m11} and the definition of $U_n$, we have
$$
\bigg[ \varepsilon(1-C) + \bigg(1 - \frac 1C\bigg) d   \bigg] a_{k_m+1}^4 \le \var_\omega \overline T_{S_{k_m}}
\le \bigg[ \varepsilon + \varepsilon(1+C) + \bigg(1 + \frac 1C\bigg) D   \bigg] a_{k_m+1}^4
$$
and
$$
\bigg[ o(1) (1-C) + \bigg(1 - \frac 1C\bigg) b   \bigg] a_{k_m+1}^4 \le \var_\omega \overline \T_{{k_m +1}}
 \le \bigg[  o(1) (1+C) + \bigg(1 + \frac 1C\bigg) B   \bigg] a_{k_m+1}^4.
$$
The above inequality ensures that $\var_\omega \overline T_{ S_{k_m+1}}$ is of the order $a^4_{k_m+1}$.
For arbitrary small $\delta>0$, choosing appropriate parameters $\varepsilon, C$  and large indices $k_m$
\begin{equation}\label{eq:m20}
(1-\delta) \frac{d}{D+B} \le   v_{m}^2 \le  (1+\delta) \cdot \frac{D}{d+b}
\end{equation}
and
\begin{equation}\label{eq:m21}
(1-\delta) \frac{b}{D+B} \le   w_{m}^2 \le  (1+\delta) \cdot \frac{B}{d+b}
\end{equation}
We can pass with $\delta$ to 0 and assume that the parameters satisfy
\begin{equation}\label{eq:m20t}
 \frac{d}{D+B} \le \liminf_{n\to\infty}  v_{m}^2 \le\limsup_{n\to\infty}  v_{m}^2 \le  \frac{D}{d+b}
\end{equation}
and
\begin{equation}\label{eq:m21t}
 \frac{b}{D+B} \le  \liminf_{n\to\infty} w_{m}^2 \le\limsup_{n\to\infty} w_{m}^2 \le \frac{B}{d+b}.
\end{equation}

Observe		that  for any $\eta > 0$, using the Chebyshev inequality and \eqref{eq:t2}, we have:% and the fact that $a_n$ is regularly varying with index $1/\beta$,
		\begin{multline*}
			\P_\omega [|Z_m| > \eta] =
			\P\left[ \left|  \overline \T^l_{k_m+1} - \E_\omega \overline \T^l_{k_m+1}\right| > \eta \sqrt{\var_\omega  \overline T_{S_{k_m+1}} }\; \right]
			\le \frac{ C a^3_{k_m+1}}{\eta^2 \var_\omega  \overline T_{S_{k_m+1}}}  \to 0.
		\end{multline*}
		This proves
		\begin{equation}\label{eq:m12}
			Z_m \overset{\Po}{\to} 0, \quad m\to \infty.
		\end{equation}
%		Using the estimates in the definition of the event $U_n (d,D,b,B,\varepsilon)$ and \eqref{eq:m4} gives,
%		\begin{equation}\label{eq:m20}
%			 \left(\frac{d}{ d+B} \right)^{2(\beta\wedge 1)/\beta}\le   v_{m} \le  \left(\frac{ D}{ D+b} \right)^{2(\beta\wedge 1)/\beta}
%		\end{equation}
%		and
%		\begin{equation}\label{eq:m21}
%			\frac{b^2}{D^{2(\beta\wedge 1)/\beta}} \le   w_{m}a_{k_m+1}^{-2(1-1/\beta)} \le  \frac{B^2}{ d^{2(\beta\wedge 1)/\beta}}.
%		\end{equation}

One can easily see that for any fixed $d$ and $D$ one can construct  sequences $\{b_m\}, \{B_m\}, \{k_m\}$ such that $b_m,B_m\to \infty$,
 $b_m/B_m\to 1$
and inequalities \eqref{eq:m20t} and $\eqref{eq:m21t}$ hold. Then $v_{m}\to 0$ and $w_{m} \to 1$. Since %, in view of \eqref{eq:mm3}
the sequence $\{V_m\}$ is tight, we have
$$
			\frac{\overline T_{S_{k_m+1}} - \E_\omega \overline T_{S_{k_m + 1}}} {{\sqrt{\var_\omega \overline T_{S_{k_m+1}}}}}
				= v_{m} \cdot V_m +w_m \cdot W_m + Z_m \Rightarrow 2\vartheta-1.
$$ So, if the limits \eqref{eq:m6}, \eqref{eq:mm3} exist, both must be equal to $Y_\omega = 2\vartheta-1$.

Next, fixing all the parameters $b,B,d,D$ observe that
		both sequences $\{v_{m}\}$, $\{w_m\}$ are bounded, therefore we can assume, possibly choosing their subsequences,
		that they are convergent to some $v$ and $w$, respectively.
		Since the sequences of random variables $\{V_m\}$ and $\{W_m\}$ are independent,
 we conclude
		\begin{equation}\label{eq:t9}
			2\vartheta-1 \overset{d}{=} Y_\omega \overset{d}{=} v (2\vartheta_v-1) + w (2\vartheta_w-1),
		\end{equation}
		where $\vartheta_v$, $\vartheta_w$ are independent. However this equation cannot be satisfied e.g. by \eqref{eq:mm1}. That leads
		to a contradiction and proves that the limit \eqref{eq:m6} cannot exist.

\medskip

Summarizing, we have proved up to now that the sequence $\{\widetilde T_n\}$ defined in \eqref{eq:t7} cannot converge in distribution.
Nevertheless, it still can happen that different normalization leads to a convergent sequence and we need to exclude this possibility.
Let us consider another normalization $\widehat{T}_n = \widetilde{T}_n/A_n + C_n$ for some sequences $\{A_n\}$ and $\{C_n\}$.
Let us also recall that we already know that if the limit exists, it must be equal to $2\vartheta - 1$.

Observe first that $A_n$ must be bounded. Indeed, assume that there exists its subsequence $\{A_{n_k}\}$ converging to $+\infty$.
Then, since $\{\widetilde{T}_n\}$ is tight, we have $\widetilde{T}_{n_k}/A_{n_k} \overset{\P}{\to} 0$ and thus the sequence $\{C_{n_k}\}$
must converge to some limit $C$ and finally $\widehat{T}_{n_k} \Rightarrow C$, which is a trivial limit.

Next we fix parameters $b,B,d,D$ and consider the subsequence $\{k_m\}$ satisfying all the above requirements.
If the sequence $\{A_n\}$ contains a subsequence $\{A_{k_m}\}$ convergent to a positive constant $A$, then again the
sequence $\{C_{k_m}\}$ must converge to some $C$. Repeating the above calculations we are led once more to equation \eqref{eq:t9},
which leads us to a contradiction.

Thus, the sequence $\{A_{k_m}\}$ must converge to 0. However in this case, since $w_nW_m + Z_m \Rightarrow w (2\vartheta-1)$, the sequence $\{(w_mW_m + Z_m)/A_{k_m}\}$ cannot be tight
and finally, since $V_n$ is independent,  by \eqref{eq:m7} the sequence $\widehat{T}_{k_m}$ cannot be tight and converge in distribution.
This completes the proof of the theorem.

	\end{proof}

\section{Proofs of the weak quenched limit laws}\label{sec:proof}

	In this final section we present a complete proof of our main results. We will begin by presenting a suitable coupling. Then we will treat the moderately sparse and strongly sparse case separately.

\subsection{Coupling}\label{subsec:coupling}

	In the first step we will prove our result along the marked sites. That is we analyse
  	\begin{equation}\label{eq:5:phi}
    		\phi_{n, \omega} (\cdot ) = \P_\omega \left[ { a_n^{- 2}} \left( T_{S_n} - \E_\omega [T_{S_n}]  \right) \in \: \cdot \: \right].
  	\end{equation}
	The main part of the argument concentrates on the limit law of $T^r_{S_n} = \sum_{k=1}^n \T_k^{r}$. Recall $U_n$ defined in~\eqref{eq:3:defUn}, which is
	the first time the reflected random walk hits $n$. For every $k>0$ and for fixed environment $\omega$ it holds $\T_k^{r}\overset{d}{=} U_{\xi_k}$.
	By the merit of Lemma~\ref{lem:m3} and Skorokhod's representation theorem we may assume that our space holds random variables $U_n^{(k)}$ and $\vartheta_k$ such that:
 	\begin{itemize}
 		\item $\{U_n^{(k)}\}_n, \vartheta_k$ for $k \in \N$, are independent copies of $\{U_n\}_n, \vartheta$;
 		\item $\{U_n^{(k)}, \vartheta_k : n, k \in \N\}$ and $\{\xi_k : k \in \N\}$ are independent;
 		\item $U_n^{(k)}/n^2 \to 2\vartheta_k$ in $L^2$ as $n \to \infty$;
 		\item for all $\omega$, $U_{\xi_k}^{(k)}$ and $\T_k^r$ have the same distribution under $\Po$.
 	\end{itemize}
 Observe that the convergence in $L^2$ is secured by the convergence in distribution and uniform integrability provided in Lemma \ref{lem:m3}.

 	To simplify the notation we will write $U_{\xi_k}$ instead of $U_{\xi_k}^{(k)}$.

 	\begin{prop}\label{prop:4:right}
 		Assume~\eqref{eq:2:assumption}. Then as $n \to \infty$,
 		\begin{equation*}
 			a_n^{-4} \var_\omega \left[  T_{S_n}^r - \E_\omega T_{S_n}^r - \sum_{k=1}^n \xi_k^2 (2\vartheta_k-1) \right] \overset{\P}{\to} 0.
 		\end{equation*}
 	\end{prop}

 	\begin{proof}
		Firstly note that
		\begin{equation*}
			\var_\omega \left[  T_{S_n}^r - \E_\omega T_{S_n}^r - \sum_{k=1}^n \xi_k^2 (2\vartheta_k-1) \right] = \var_\omega \left[ \sum_{k=1}^n \left( U_{\xi_k} - 2\xi_k^2 \vartheta_k \right) \right] .
		\end{equation*}
  For $\varepsilon > 0 $ let $I^1_n = \{k \leq n : \xi_k > \varepsilon a_n\}$ and $I^2_n = \{k \leq n : \xi_k \le \varepsilon a_n\}$. Then for any $\delta>0$,
 	\begin{multline}	
 		\P\left[ \var_\omega \left[ \sum_{k=1}^n \left( U_{\xi_k} - 2\xi_k^2 \vartheta_k \right) \right] > \delta a_n^4 \right] \leq \\
 		 \P\left[  \sum_{k\in I^1_n } \xi_k^4 \var_\omega \left[  \frac{U_{\xi_k}}{\xi_k^2} - 2\vartheta_k \right] > \frac{ \delta a_n^4}{2} \right] + \P\left[  \sum_{ k\in I^2_n } \xi_k^4 \var_\omega \left[  \frac{U_{\xi_k}}{\xi_k^2} - 2\vartheta_k \right] > \frac{ \delta a_n^4}{2} \right]. \label{eq:5:cos}
 	\end{multline}
 	Since $U_n^{(k)}, \vartheta_k$ are independent copies of $U_n, \vartheta$ such that $U_n / n^2 \to \vartheta$ in $L^2$, there exists $M > 0$ such that
 	\begin{equation*}
 		\var_\omega \left[  \frac{U_{\xi_k}}{\xi_k^2} - 2\vartheta_k \right] < M \quad \text{ for all } k, \omega,
 	\end{equation*}
 	and, moreover, for $N \in \N$ large enough
 	\begin{equation*}
 		\var_\omega \left[  \frac{U_{N}^{(k)}}{N^2} - 2\vartheta_k \right] < \varepsilon \quad \text{ for all } k, \omega.
 	\end{equation*}
 	We can hence estimate, for $n$ sufficiently large,
 	\begin{equation*}
 		\P\left[  \sum_{k\in  I^1_n } \xi_k^4 \var_\omega \left[  \frac{U_{\xi_k}}{\xi_k^2} - 2\vartheta_k \right] > \frac{ \delta a_n^4}{2} \right] \leq
		\P\left[ \frac{ \sum_{k=1}^n \xi_k^4}{ a_n^4}  > \frac{ \delta }{2\varepsilon} \right].
 	\end{equation*}
Since  the sequence $ \sum_{k=1}^n \xi_k^4 / a_n^4$ converges weakly (under $\P$) to some $\beta/4$-stable variable $L_{\beta/4}$,
	the probability on the right hand side above converges to $\P[L_{\beta/4} >\delta /(2\varepsilon) ]$.
 	To estimate the second term in~\eqref{eq:5:cos}, note that
 	\begin{multline*}
 		\P\left[  \sum_{ k\in  I^2_n } \xi_k^4 \var_\omega \left[  \frac{U_{\xi_k}}{\xi_k^2} - 2\vartheta_k \right] > \frac{ \delta a_n^4}{2} \right]
 			\leq \P\left[  \sum_{k=1}^n \xi_k^4 \1_{\{\xi_k \leq \eps a_n\}} > \frac{ \delta a_n^4}{2M} \right] \\
 			\leq \frac{2 M}{\delta} a_n^{-4} \E\left[  \sum_{k=1}^n \xi_k^4 \1_{\{\xi_k \leq \eps  a_n\}}\right]
 			= \frac{2M}{\delta} n a_n^{-4}\E\left[\xi_1^4 \1_{\{\xi_1 \leq \eps a_n\}}\right].
 	\end{multline*}
By the Fubini theorem, we have
 	\begin{equation*}
 		\E\left[\xi_1^4 \1_{\{\xi_1 \leq \eps a_n \}}\right] \leq  \int_0^{\eps a_n } 4 t^3 \P[\xi_1 > t]  dt
 	\end{equation*}
and the Karamata  theorem~\cite[Theorem 1.5.11]{bingham:1987:regular} entails that the expression on the right is
asymptotically equivalent to $4\varepsilon^4 a_n^4 \P[\xi_1 > \varepsilon a_n] \sim 4 \varepsilon^{4-\beta} n^{-1}a_n^4$.
Finally, we can conclude that for any $\eps, \delta > 0$,
 	$$ \limsup_n\P\left[  \sum_{k=1}^n \xi_k^4 \var_\omega \left[  \frac{U_{\xi_k}}{\xi_k^2} - \vartheta_k \right] > \delta a_n^4 \right] \leq
 \frac{8M}{\delta} %\frac{8M}{\delta(4-\beta)}
   \eps^{4-\beta} + \P\left[L_{\beta/4} > \frac{\delta}{2\eps}\right] $$
and passing with $\varepsilon$ to 0 we conclude the desired result.
 \end{proof}

	We are now ready to determine the weak limit of the sequence $\phi_n(\omega) = \phi_{n,\omega}$ given by \eqref{eq:5:phi}. Recall the map $G$ defined in \eqref{eq:G}.

\begin{lem}\label{lem:4:meas}
	The map $G$ is measurable.
\end{lem}

\begin{rem}\label{rmk:4:1}
	The proof of Lemma \ref{lem:4:meas} is identical to that of Lemma 1.2 in \cite{peterson:2013:weak} and therefore will be omitted. Part of the proof is showing that the map
	\begin{equation*}
		G_2 : \ell^2 \ni (x_k)_{k \in \N} \mapsto \PP\left[ \sum_{k=1}^\infty x_k(2\vartheta_k - 1) \in \cdot \right] \in \mathcal{M}_1(\R)
	\end{equation*}
	is continuous.
\end{rem}

\begin{thm}\label{thm:4:Sn}
	Assume~\eqref{eq:2:assumption} and~\eqref{eq:2:assumption2}. Then
	\begin{equation*}
	\phi_n\Rightarrow G(N_\infty),
	\end{equation*}
	in $\mathcal{M}_1$,
	where $N_\infty$ is a Poisson point process with intensity $\beta  x ^{-\beta/2-1} \ud x/2$.	
\end{thm}

	In the proof of this result we will use the following lemma.
%{\color{ala} \tt ten lemat podmienic na uwage o zbieznosci sum i wariancji, ktora wystepuje po nim w cytowanej pracy, bo i tak tylko z niej korzystamy}
\begin{lem}[{\cite[Remark 3.4]{peterson:2013:weak}}]\label{lem:3:reduction}
  Let $\theta_n$ be a sequence of random probability measures on $\RR^2$ defined on the same probability space. Let $\gamma_n$ and $\gamma_n'$ denote the marginals of $\theta_n$. Suppose that
  \begin{equation*}
    \E_{\theta_n} (X-Y) \overset{\P}{\to} 0 \quad \mbox{and} \quad \var_{\theta_n}(X-Y) \overset{\P}{\to} 0,
  \end{equation*}
  where $X$ an $Y$ are the coordinate variables in $\RR^2$. If $\gamma_n \Rightarrow \gamma$, then $\gamma_n' \Rightarrow \gamma$.
\end{lem}

\begin{proof}[Proof of Theorem~\ref{thm:4:Sn}]
	First, observe that the sequence of random point measures $N_n =\sum_{k=1}^n \delta_{\xi^2_k a_n^{-2}}$ converges weakly to $N_\infty$. Indeed, this follows by an appeal to
	 \cite[Proposition 3.21]{resnick:2013:extreme} and checking that
	\begin{equation*}
		n\P[\xi^2 /a_n^2 \in \cdot] \to \mu(\cdot) \quad \text{vaguely on } (0,\infty]
	\end{equation*}
	where $\mu(\ud x) = \beta x^{-\beta/2-1} \ud x/2$.

Since $G$ is not continuous, we cannot  simply apply the continuous mapping theorem  and, similarly as in \cite{peterson:2013:weak}, we are forced to follow a more tedious argument.
%	Next we would like to show that $H(N_n) \Rightarrow H(N_\infty)$. Were $H$ continuous, it would be a direct consequence of the continuous mapping theorem. Unfortunately,
%	as is the case in a similar problem in \cite{peterson:2013:weak}, $H$ is not continuous and we must use a longer argument.
	Define
	$$
		G_\varepsilon : \mathcal{M}_p((0,\infty]) \ni \sum_{k=1}^\infty \delta_{x_k} \mapsto
		\PP\left[ \sum_{k=1}^\infty x_k(2\vartheta_k -1) \1_{ \{ x_k > \varepsilon\} } \in \cdot \right] \in \mathcal{M}_1(\R).
	$$
	Then for any $\varepsilon > 0$ the map $G_\varepsilon$ is continuous on the set $\mathcal M_p^\varepsilon := \left\{ \zeta \in \mathcal M_p : \zeta(\{\varepsilon , \infty\}) = 0 \right\}$;
	indeed, take $\zeta_n, \zeta \in \mathcal M_p^\varepsilon$ such that $\zeta_n \to \zeta$ vaguely. Then, by \cite[Proposition 3.13]{resnick:2013:extreme}, since the set
	$[\varepsilon, \infty]$ is compact in $(0,\infty]$, there exists $p_\varepsilon < \infty$ and an enumeration of points of $\zeta$ and $\zeta_n$ (for $n$ sufficiently large) such that
	\begin{equation*}
		\zeta_n(\cdot \cap [\varepsilon, \infty]) = \sum_{k=1}^{p_\varepsilon} \delta_{x_k^n}, \quad \zeta(\cdot \cap [\varepsilon, \infty]) = \sum_{k=1}^{p_\varepsilon} \delta_{x_k}
	\end{equation*}
	and
	\begin{equation*}
		(x_1^n, \dots, x_{p_\varepsilon}^n) \to (x_1, \dots, x_{p_\varepsilon}) \quad \textnormal{as } n\to\infty.
	\end{equation*}
	Therefore
	\begin{equation*}
		G_\varepsilon (\zeta_n) (\cdot)= \PP\left[ \sum_{k=1}^{p_\varepsilon} x_k^n(2\vartheta_k -1) \in \cdot \right] \Rightarrow
		\PP\left[ \sum_{k=1}^{p_\varepsilon} x_k(2\vartheta_k -1) \in \cdot \right] = G_\varepsilon(\zeta)(\cdot).
	\end{equation*}
	By~\cite[Theorem 3.2]{billingsley1999convergence}, to prove that $G(N_n) \Rightarrow G(N_\infty)$ it is enough to show
	\begin{align}
	G_\varepsilon(N_n) \Rightarrow_n G_\varepsilon(N_\infty)& \quad &\textnormal{for all } \varepsilon > 0, \label{eq:4:cond1} \\
	G_\varepsilon(N_\infty) \Rightarrow_\varepsilon G(N_\infty)&, \label{eq:4:cond2} \\
	\lim_{\varepsilon \to 0}\limsup_{n \to \infty} \P\left[ \rho (G_\varepsilon(N_n), G(N_n)) > \delta \right] = 0& &\textnormal{for all } \delta > 0,\label{eq:4:cond3}
	\end{align}
	where $\rho$ is the Prokhorov metric on $\mathcal{M}_1(\R)$.
	
	First, for any $\varepsilon > 0$, $N_\infty \in \mathcal{M}_p^\varepsilon$ almost surely. Thus \eqref{eq:4:cond1} is satisfied by the continuous mapping theorem since $G_\varepsilon$ is continuous.
	
	For any sequence $\boldsymbol{x} = (x_k)_{k\in\N} \in \ell^2$ and $\varepsilon > 0$ define $\boldsymbol{x}^\varepsilon \in \ell^2$ by $x_k^\varepsilon = x_k
\1_{\{ x_k > \varepsilon\}}$. By the dominated convergence theorem, $\boldsymbol{x}^\varepsilon \to \boldsymbol{x}$ in $\ell^2$ as $\varepsilon \to 0$. Hence, since the map $G_2$ defined in Remark \ref{rmk:4:1} is continuous, also $G_2(\boldsymbol{x}^\varepsilon) \Rightarrow G_2(\boldsymbol{x})$. This means that for any point process $\zeta = \sum_k \delta_{x_k}$ such that $\boldsymbol{x} \in \ell^2$,
	$$ G_\varepsilon(\zeta) = G_2(\boldsymbol{x}^\varepsilon) \Rightarrow G_2(\boldsymbol{x}) = G(\zeta),$$
	which gives \eqref{eq:4:cond2}.
	
	Recall that if $\mathcal L_X, \mathcal L_Y$ are laws of random variables $X,Y$ defined on the same probability space, then $\rho(\mathcal L_X, \mathcal L_Y)^3 < \E|X-Y|^2$ (c.f.~\cite[Theorem 11.3.5]{dudley2018real}). Thus
	\begin{equation*}
	\begin{split}
	\P\left[ \rho (G_\varepsilon(N_n), G(N_n)) > \delta \right]
	&\leq \P \left[ \Eo\left| a^{-1}_n \sum_{k=1}^n {\xi_k}{} \1_{\{\xi_k\leq \varepsilon a_n \}}(2\vartheta_k - 1) \right|^2 > \delta^3 \right] \\
	& = \P \left[ \EE(2\vartheta_1 - 1)^2 a^{-2}_n\sum_{k=1}^n {\xi_k^2} \1_{\{\xi_k\leq \varepsilon a_n\}} > \delta^3 \right],
	\end{split}
	\end{equation*}
	since $(2\vartheta_k-1)_k$ is a sequence of mean $0$ i.i.d. variables independent of the environment. Denote $C = \EE(2\vartheta_1-1)^2$, then
	\begin{equation*}
	\begin{split}
	\limsup_{n\to\infty} \P\left[ \rho (G_\varepsilon(N_n), G(N_n)) > \delta \right]
	&\leq \limsup_{n\to\infty} \P \left[ {a_n^{-2}} \sum_{k=1}^n {\xi_k^2} \1_{\{\xi_k\leq \varepsilon a_n\}} > \frac{\delta^3}{C} \right] \\
	&\leq \limsup_{n\to\infty} \P \left[ { a^{-2}_n}\sum_{k=1}^n {\xi_k \varepsilon a_n} \1_{\{\xi_k\leq \varepsilon a_n \}} > \frac{\delta^3}{C} \right] \\
	&\leq \lim_{n\to\infty} \P \left[ {a^{-1}_n}\sum_{k=1}^n {\xi_k} > \frac{\delta^3}{\varepsilon C} \right].
	\end{split}
	\end{equation*}
	The sequence $a^{-1}_n\sum_{k=1}^n \xi_k$ converges weakly to some $\beta$-stable variable $L_\beta$, therefore
	\begin{equation*}
	\lim_{n\to\infty} \P \left[ a^{-1}_n \sum_{k=1}^n {\xi_k} > \frac{\delta^3}{\varepsilon C} \right] = \P\left[ L_\beta > \frac{\delta^3}{\varepsilon C} \right] \to 0 \quad \textnormal{as } \varepsilon \to 0,
	\end{equation*}
	which proves \eqref{eq:4:cond3}.
	
	Therefore $G(N_n) \Rightarrow G(N_\infty)$. Now the claim of the theorem follows from Proposition~\ref{prop:4:right} and Lemmas \ref{lem:3:reduction} and~\ref{lem:var}.
\end{proof}

\subsection{Moderately sparse random environment}

	\begin{proof}[Proof of Theorem~\ref{thm:2:m1}]
%{\red \tt to do:
%\begin{itemize}
%\item poprawic normalizacje
%\item $\phi_n^\alpha$, czy  $\phi_{n,\omega}^\alpha$
%\item co sie stalo z $T^l$??? To jest pominiete w dowodzie, a w tym celu jest lemat \ref{lem:var} oraz
%byl inny lemat o slabej zbieznosci i wariancji, ktory zostal usuniety.
%\end{itemize}}
		Since $\E\xi_1 < \infty$, $\alpha = (\E\xi_1)^{-1}$ is well defined.
Let $N_\infty = \sum_n \delta_{x_n}$ be  a Poisson point process as in Theorem~\ref{thm:4:Sn} and let
		$N_\infty^\alpha = \sum_n \delta_{\alpha^{2/\beta}x_n}$. Then $N_\infty^\alpha$ is a Poisson point process with intensity $\beta\alpha x^{-\beta/2-1}\ud x/2$.
Putting
\begin{align*}
  \phi_n^\alpha(\omega)(\cdot) &= \phi_{n,\omega}^\alpha (\cdot) = \Po\left[ a_n^{-2}(T_{S_{\alpha n}} - \Eo T_{S_{\alpha n}} ) \in \cdot \right]\\
  & = \Po\left[(a_{\alpha n} /a_n)^{2}a_{\alpha n}^{-2}(T_{S_{\alpha n}} - \Eo T_{S_{\alpha n}} ) \in \cdot \right],
\end{align*}
		where $S_{\alpha n} := S_{\lfloor \alpha n \rfloor}$,
		it follows from Lemma \ref{lem:3:reduction} and Theorem \ref{thm:4:Sn} that $\phi_n^\alpha \Rightarrow G(N_\infty^\alpha)$.

		It remains to show that
		\begin{equation*}
			a_n^{-4} \var_\omega\left[ (T_{S_{\alpha n}} - \Eo T_{S_{\alpha n}}) - (T_n - \Eo T_n) \right] = a_n^{-4} \var_\omega\left[ T_{S_{\alpha n}} - T_n \right] \overset{\P}{\to} 0,
		\end{equation*}
		from which it follows, by Lemma \ref{lem:3:reduction}, that $\mu_n \Rightarrow G(N_\infty^\alpha)$.

		Observe that on the event $\{ n \leq S_{\alpha n}\}$, for any $k$ such that $S_k \leq n$,
		\begin{align*}
			\var_\omega\left[ T_{S_{\alpha n}} - T_n \right] &= \sum_{j=n+1}^{S_{\alpha n}} \var_\omega \left[T_j - T_{j-1}\right]
			\leq \sum_{j={S_k} +1}^{S_{\alpha n}} \var_\omega \left[T_j - T_{j-1}\right]  \\ & = \var_\omega \left[ T_{S_{\alpha n}} - T_{S_{k}}\right]
		\end{align*}
		and similarly on $\{S_{\alpha n} \leq n\}$ for any $k$ such that $S_k \geq n$,
		\begin{align*}
			\var_\omega\left[ T_{S_{\alpha n}} - T_n \right] \leq \var_\omega \left[ T_{S_{k}} - T_{S_{\alpha n}}\right].
		\end{align*}
		Therefore  for any $\delta > 0$ and $\eps > 0$,
		\begin{align*}
			\P\left[a_n^{-4} \var_\omega\left[ T_{S_{\alpha n}} - T_n \right] > \delta \right]
				&\leq \P\left[ | S_{\alpha n} - n| > \eps n\right]  \\
				&+ \P\left[ a_n^{-4} \var_\omega \left[T_{S_{\lfloor\alpha n\rfloor + \lfloor \eps n \rfloor}} - T_{S_{\alpha n}}\right] > \delta \right] \\
				&+ \P\left[ a_n^{-4} \var_\omega \left[T_{S_{\alpha n}} - T_{S_{\lfloor\alpha n\rfloor - \lfloor \eps n \rfloor}}\right] > \delta \right] \\
				&= \P\left[ \left|\frac{S_{\alpha n}}{\alpha n} - \frac{1}{\alpha}\right| > \frac{\eps}{\alpha} \right]
					+ 2\P\left[ a_n^{-4} \var_\omega \left[ T_{S_{\eps n}} \right] > \delta\right].
		\end{align*}
		The first term tends to $0$ by the law of large numbers (recall $1/\alpha = \E\xi_1$). To estimate the second, note that
	 	\begin{equation*}
			\var_\omega \left[ T^r_{S_{\eps n}} \right] = \sum_{k=1}^{\eps n} \var_\omega T^r_{\xi_k} \overset{\eqref{eq:tvar}}{=} \sum_{k=1}^{\eps n} \frac{2}{3} (\xi_k^4 - \xi_k^2) \leq \sum_{k=1}^{\eps n} \xi_k^4,
		\end{equation*}
		%hence
and $a_n^{-4} \sum_{k=1}^{\eps n} \xi_k^4 \Rightarrow \eps^{-4/\beta} L_{\beta/4}$ with respect to $\P$, while by Lemma \ref{lem:var}, $a_n^{-4} \var_\omega\left[T_{S_{\eps n}}^l \right] \overset{\P}{\to} 0$. Therefore
		\begin{equation*}
			 \limsup_{n\to \infty} \P\left[ a_n^{-4} \var_\omega \left[ T_{S_{\eps n}} \right] > \delta\right]
			%\leq \P\left[ \frac{ \sum_{k=1}^{\eps n} \xi_k^4} { (\eps n)^{4/\beta}} > \frac{\delta}{\eps^{4/\beta}}\right] \to
			\leq \P\left[ L_{\beta/4} > \frac{\delta}{\eps^{4/\beta}}\right].
		\end{equation*}
		The last expression can be made arbitrary small by taking sufficiently small $\eps$.
	\end{proof}

\subsection{Strong sparsity: preliminaries}
From now we assume that $\E \xi =\infty$. This case is technically more involved, however the underlying principle remains the same.
   	Denote the first passage time of $S$ via
  	\begin{equation*}
    		\nu_n =\inf\left\{ k >0 \: : \: S_k>n\right\}.
  	\end{equation*}
  %	This case requires more care than the previous one. However, the underlying principle remains the same.
%   	Denote the first passage time of $S$ via
%  	\begin{equation*}
%    		\nu_n =\inf\left\{ k >0 \: : \: S_k>n\right\}.
%  	\end{equation*}
%	Note that by renewal theorem $\nu_n/n \to 1/\E \xi$ a.s. and as a consequence the asymptotic of $\nu_n$ is well described in the case $\E \xi <\infty$ as seen in the previous %section. From now we will assume that $\E \xi =\infty$.
%	In this case the behaviour of $\nu_n$ is more involved.
Recall that we write
	\begin{equation*}
		m_n = n\E \left[ \xi { \bf 1}_{\{ \xi \leq a_n \}} \right] \quad \mbox{and} \quad d_n = 1/\P[\xi>n].
	\end{equation*}
	and we denote by $\{ c_n\}_{n \in \NN}$ for the asymptotic inverse of $\{m_n\}_{n \in \NN}$, i.e. any increasing sequence of real numbers such that
	\begin{equation*}
		\lim_{n \to \infty}c_{m_n}/n = \lim_{n \to \infty}m_{c_n}/n = 1.
	\end{equation*} 
	As one may expect $S_n$ grows at a scale $m_n$ and thus $\nu_n$ must grow at a scale $c_n$ (in the sense of limit theorem which we will soon make precise). For our purposes we need to justify that $S_n/m_n$ and $\nu_n/c_n$
	converge jointly with some other characteristics of the trajectory of $S$. For this reason we will need to use the setting of c\`adl\`ag functions.
	Recall that $\mathbb{D}$ stands for the space of right continuous functions that have a left limit at each point.
 For $h \in \mathbb{D}$  we define $h^{\leftarrow} \in \mathbb{D}$ via
	\begin{equation*}
		h^{\leftarrow}(t) = \inf \{ s \: : \:h(s) > t \}.
	\end{equation*}
	%Then the map $h \mapsto h^{\leftarrow}$ is continuous in $M_1$ topology~\cite{whitt:1971}.
%	In what follows, we will use notation introduced in \cite{henry:2011}. For $h \in \mathbb{D}$ let $h^-$ be the lcrl (left-continuous, having right-hands limits) version of $h$, that is,
%	$h^-(t) = \lim_{\eps \to 0^+} h(t-\eps)$ and $h^-(0) = 0$. Similarly, let $h^+$ denote rcll version of a lcrl path. Let $\Phi: \mathbb{D}^\uparrow \to \mathbb{D}$ be given by
%	$$
%	\Phi(h) = (h^- \circ (h^{\leftarrow})^-)^+.
%	$$
	%By \cite{henry:2011}, $\Phi$ is $J_1$-continuous at $\mathbb{D}^{\uparrow\uparrow} \subset \mathbb{D}$, the set of strictly increasing, unbounded functions.
	Consider $\mathbb{D}^{\uparrow} \subseteq \mathbb{D}$ consisting of non-decreasing functions and take
	$M \colon \mathbb{D}^\uparrow \to \mathcal{M}$ given via
	\begin{equation*}
		M(h)  = \sum_{k}\delta_{x_k}\otimes\delta_{t_k},
	\end{equation*}
	where for $h \in \mathbb{D^\uparrow}$, $\{t_k\}$ are the discontinuity points of $h$ and $x_k = h(t_k)-h(t_k^-)$ is the size of the jump at $t_k$.
	
	\begin{lem}\label{lem:5:M}
		The function $M \colon \mathbb{D}^{\uparrow}\to  \mathcal{M}$ is continuous with respect to $J_1$ topology.
	\end{lem}
	\begin{proof}
		Let $f_n, f \in  \mathbb{D}^{\uparrow}$ be such that $f_n \to f$ in $J_1$ topology. For any nonnegative, continuous function $\varphi \colon (0, +\infty]\times [0, +\infty) \to \RR$ with compact support we can find $\varepsilon>0$ and
		$T>0$ such that	$\varphi(x,t) =0$ if $x\leq\varepsilon$ or $t\geq T$. Since $f \in \mathbb{D}^{\uparrow}$, it has only finitely many jumps on the interval $[0, T]$ that are greater than $\varepsilon$,
		\begin{equation*}
			\int \varphi(x,t) \: Mf(\ud x, \ud t) = \sum_{k=1}^N \varphi(x_k, t_k)
		\end{equation*}
		 for some $N$,  $t_1 < \dots < t_N < T$ and $x_k > \varepsilon$.
		
		By the definition of $J_1$ topology, there exists a sequence of continuous increasing functions $\lambda_n : (0,\infty) \to (0,\infty)$ such that
		\begin{equation}\label{eq:4:J1}
			\sup_{t \in [0,T]} |\lambda_n(t) - t| \to 0, \quad \sup_{t \in [0,T]} |f_n(t) - f(\lambda_n(t)) | \to 0.
		\end{equation}
		For $n$ sufficiently large, $\sup_{t \in [0,T]} |\lambda_n(t) - t| < T-t_N$, which means that $f \circ \lambda_n$ has exactly $N$ jumps on the interval $[0,T)$, at times $\lambda_n^{-1}(t_k)$. Moreover, for large enough $n$,
		$\sup_{t \in [0,T]} |f_n(t) - f(\lambda_n(t)) | < \varepsilon/3,$
		from which it follows that $f_n$ cannot have jumps bigger than $\varepsilon$ apart from the discontinuity points of $f \circ \lambda_n$.
		
		Fix $k \in \{1, \dots N\}$. It follows from \eqref{eq:4:J1} that for $n$ large enough $f_n$ does have a jump at $\lambda_n^{-1}(t_k)$, denote it by $x_k^n$, and observe that $x_k^n \to x_k$ as $n\to \infty$; in particular
		$x_k^n > \varepsilon$ for large $n$. It also follows that $\lambda_n^{-1}(t_k) \to t_k$ as $n\to\infty$. This means that for $n$ sufficiently large
		\begin{equation*}
			\int \varphi(x,t) \: Mf_n(\ud x, \ud t) = \sum_{k=1}^N \varphi(x^n_k, \lambda_n^{-1}(t_k))
		\end{equation*}
		and the last expression tends to $\int \varphi(x,t) \: Mf(\ud x, \ud t)$ as $n\to \infty$, which gives $Mf_n \to Mf$.
	\end{proof}

	Consider a random element of $\mathcal{M}_1((0, +\infty]\times [0, +\infty))$ given by
	\begin{equation*}
		\Lambda_n = \sum_{j=1}^\infty \delta_{\xi_j/a_n} \otimes \delta_{j/n}
	\end{equation*}
	and random elements of $\mathbb{D}^\uparrow$ defined via
	\begin{equation}\label{eq:5:sub}
		L_n(t) = S_{  \lfloor nt \rfloor}/ a_n \mbox{ for } \beta< 1, \quad \mbox{and} \quad
			\quad \widetilde{L}_n(t) = S_{ \lfloor nt \rfloor}/m_n \mbox{ for } \beta = 1.
	\end{equation}
	Recall $\Upsilon \colon \mathbb{D}^\uparrow \to \RR$ defined in~\eqref{eq:2:upsilon}.

	\begin{lem}\label{lem:5:conv}
		If $\beta < 1$, then
		\begin{equation}\label{eq:5:conv<1}
			\left(L_n, \Lambda_n, \frac{\nu_n}{d_n} , \frac{ S_{\nu_n-1}}{n}\right) \Rightarrow (L,M(L),L^{\leftarrow}(1),\Upsilon(L))
		\end{equation}
		in $ (\mathbb{D}, J_1) \times \mathcal{M}\times \RR\times \RR$, where $L = (L_t)_{t \geq 0}$ is strictly increasing $\beta$-stable subordinator with L\'evy measure given by $\nu(x, +\infty) =  x^{-\beta} $.
		
		If $\beta = 1$, then
		\begin{equation}\label{eq:5:conv=1}
			\left(\widetilde{L}_n, \frac{\nu_n}{c_n} , \frac{ S_{\nu_n-1}}{n}\right) \Rightarrow ( {\rm id},1,1)
		\end{equation}
		in $ (\mathbb{D}, J_1)\times \RR\times \RR$, where ${\rm id} \colon \RR_+ \to \RR_+$ is the identity function.
	\end{lem}
	\begin{proof}
		Consider first $\beta < 1$. By an appeal to standard functional weak convergence to stable L\'evy motion \cite[Corollary 7.1]{resnick2007heavy},
		$$L_n \Rightarrow L \quad \textnormal{in } (\mathbb D, J_1).$$
		Note that
		$$ \Lambda_n = M(L_n) $$
		and the function $M$ is $J_1$-continuous by Lemma \ref{lem:5:M}. %{\red ten lemat, a przynajmniej sformulowanie powinno byc wczesniej}.
		Moreover,
		$$ \frac{\nu_n}{d_n} = L_{d_n}^\leftarrow\left( \frac{n}{a_{d_n}}\right) $$
		and the map $h \mapsto h^{\leftarrow}$ is continuous in $M_1$ topology by~\cite{whitt:1971}.
		In what follows, we will use notation introduced in \cite{henry:2011}. For $h \in \mathbb{D}$ let $h^-$ be the lcrl (left-continuous, having right-hands limits) version of $h$, that is,
		$h^-(t) = \lim_{\eps \to 0^+} h(t-\eps)$ and $h^-(0) = 0$. Similarly, let $h^+$ denote rcll version of a lcrl path. Let $\Phi: \mathbb{D}^\uparrow \to \mathbb{D}$ be given by
		$$
			\Phi(h) = (h^- \circ (h^{\leftarrow})^-)^+.
		$$
		Finally, observe that for any $k\in \NN$, $\Phi(L_{d_n})$ on the set $[S_k/a_{d_n}, S_{k+1}/a_{d_n})$ is constant and equal to $S_k/a_{d_n}$, therefore
		$$\frac{S_{\nu_n - 1}}{a_{d_n}} = \Phi\left(L_{d_n}\right)\left(\frac{n}{a_{d_n}}\right). $$
		By \cite{henry:2011}, $\Phi$ is $J_1$-continuous on $\mathbb{D}^{\uparrow\uparrow} \subset \mathbb{D}$, the set of strictly increasing, unbounded functions. Since $L \in\mathbb{D}^{\uparrow\uparrow}$ almost surely,  by the continuous mapping theorem we have joint convergence in distribution
		$$ (L_n, M(L_n), L_n^\leftarrow, \Phi(L_n)) \to (L,M(L),L^\leftarrow,\Phi(L)) $$
		in $ (\mathbb{D}, J_1) \times \mathcal{M}_p((0,\infty]\times[0,\infty)) \times (\mathbb{D}, M_1) \times (\mathbb{D}, J_1)$. By Skorokhod's representation theorem we may assume the above convergence holds almost surely.
		
		Since the limiting processes admit no fixed discontinuities, Proposition 2.4 in \cite{henry:2011} gives
		$$ \frac{\nu_n}{d_n} \to L^\leftarrow(1) \quad \textnormal{and} \quad \frac{S_{\nu_n -1}}{a_{d_n}} \to \Phi(L)(1) = \Upsilon(L) $$
		almost surely.
		
		The case $\beta=1$ is similar and follows from the fact that by~\cite[Corollary 7.1]{resnick2007heavy} and properties of $J_1$ topology,
%{\red \tt id? dlaczego $t\ge 1$}
		$$\widetilde{L}_n \Rightarrow {\rm id} \quad \textnormal{in } (\mathbb D, J_1).$$
		One can combine this with
			$$ \frac{\nu_n}{c_n} = \widetilde L_{c_n}^\leftarrow\left( \frac{n}{m_{c_n}}\right), %\qquad \mbox{and} 
\qquad \frac{S_{\nu_n - 1}}{m_{c_n}} = \Phi\big(\widetilde L_{m_n}\big)\left(\frac{n}{m_{c_n}}\right) $$
	and the arguments presented in the case $\beta<1$ to get the desired claim.
	\end{proof}

%		In both cases we may assume that on our probability space the convergence holds almost surely.		

 \begin{rem} \label{rmk:as}
 	Observe that all information on the sequence $(\xi_k)_k$ is carried by the process $\Lambda_n$ and therefore by $L_n$ or, equivalently, $\widetilde{L}_n$. We may thus assume that our space holds random variables $U_n^{(k)}, \vartheta_k$ as described in Section \ref{subsec:coupling} and at the same time the convergence given in Lemma \ref{lem:5:conv} holds almost surely.
 \end{rem}

	\begin{lem}\label{lem:5:var1}
		Assume that~\eqref{eq:2:assumption} holds true.
		If  $\beta < 1$, then
		\begin{equation*}
			n^{-4} \var_\omega \left[  T_{S_{\nu_n-1}}^r - \E_\omega T_{S_{\nu_n-1}}^r - \sum_{k=1}^{\nu_n-1} \xi_k^2 (2\vartheta_k-1) \right] \overset{\P}{\to} 0.
		\end{equation*}
		If $\beta = 1$ and $\E \xi =\infty$, then
		\begin{equation*}
			a_{c_n}^{-4} \var_\omega \left[  T_{S_{\nu_n-1}}^r - \E_\omega T_{S_{\nu_n-1}}^r - \sum_{k=1}^{\nu_n-1} \xi_k^2 (2\vartheta_k-1) \right] \overset{\P}{\to} 0.
		\end{equation*}
	\end{lem}
	\begin{proof}
		One can use the same arguments as in the proof of Proposition~\ref{prop:4:right}. First consider $\beta\in (0,1)$. By tightness of $\{  \nu_n /d_n\}_{n \in \NN}$ we can choose $C>0$ to make the probability
		$\P[\nu_n > C d_n]$ arbitrarily small. Next, on the event $\{ \nu_n \leq C d_n \}$,
		\begin{equation*}
			\var_\omega \left[  T_{S_{\nu_n-1}}^r - \E_\omega T_{S_{\nu_n-1}}^r - \sum_{k=1}^{\nu_n-1} \xi_k^2 (2\vartheta_k-1) \right] \leq
			\var_\omega \left[ \sum_{k=1}^{Cd_n} \left( U_{\xi_k} - 2\xi_k^2 \vartheta_k \right) \right] .
		\end{equation*}
		From  here, since $a_{Cd_n} \sim C^{1/\beta}n$,  one argues as in the proof of Proposition~\ref{prop:4:right} to show that
 		\begin{equation*}
 			 \sum_{k\in { I^1_n}} \frac{\xi_k^4}{n^4} \var_\omega \left[  \frac{U_{\xi_k}}{\xi_k^2} - 2\vartheta_k \right]\overset{\P}{\to}0\qquad \mbox{and} \qquad  \sum_{{k\in I^2_n}} \frac{\xi_k^4}{n^4}
			\var_\omega \left[  \frac{U_{\xi_k}}{\xi_k^2} - 2\vartheta_k \right]  \overset{\P}{\to}0,
 		\end{equation*}
		where { $I^1_{n} = \{ k \leq Cd_n \: : \: \xi_k > \varepsilon n \}$, $I^2_{n} = \{ k \leq Cd_n \: : \: \xi_k \le \varepsilon n \}$} with fixed $\varepsilon>0$. In the case $\beta=1$ and $\E \xi =\infty$ one can invoke that same arguments combined with the tightness of
		$\{ \nu_n /c_n\}_{n \in \NN}$.
	\end{proof}

\begin{lem}\label{lem:4:nun}
	Assume that~\eqref{eq:2:assumption} holds true.
	If $\beta \in (0,1)$, then
	\begin{equation*}
		n^{-4} \var_\omega T_{S_{\nu_n}}^l \overset{\P}{\to} 0.
	\end{equation*}
	If $\beta=1$ and $\EE \xi =\infty$, then
	\begin{equation*}
		a_{c_n}^{-4} \var_\omega T_{S_{\nu_n}}^l \overset{\P}{\to} 0.
	\end{equation*}
\end{lem}
\begin{proof}
 	Consider the case $\beta < 1$. Take any $C>0$ and write
	\begin{equation*}
		\P\left[ n^{-4} \var_\omega T_{S_{\nu_n}}^l \geq \varepsilon \right] \leq \P\left[\nu_n \geq Cd_n \right] + \P\left[ \var_\omega T_{S_{[Cd_n ]}}^l \geq \varepsilon n^4 \right].
\end{equation*}
	Since $a_{Cd_n} \sim C^{1/\beta}n$, an appeal to Lemma \ref{lem:var} shows that the second term tends to $0$ as $n \to \infty$. The first term can be made arbitrary small by taking $C>0$ sufficiently big.
	In the case $\beta = 1$ we can use an analogous argument with $d_n$ replaced with $c_n$.
\end{proof}

	 For the purpose of the next lemma. let $(\{ U^{0}_n\}_{n \in \NN}, \vartheta_0)$ be, as before, a copy of $(\{ U_n\}_{n \in \NN},\vartheta)$ given by the claim of Lemma~\ref{lem:m3} independent of the environment.

	\begin{lem}\label{lem:5:u0}
		Assume that~\eqref{eq:2:assumption} and \eqref{eq:2:assumption2} hold true for $\beta \leq 1$ and $\E \xi =\infty$. Then
		$$
			\frac{U^0_{n-S_{\nu_n -1}} - \Eo U^0_{n-S_{\nu_n -1}}}{ n^2} - (1- \Xi)^2 (2\vartheta_0-1)  \overset{\P}{\to} 0,
		$$
		where $\Xi = \Upsilon(L)$ for $\beta < 1$ and $\Xi = 1$ for $\beta = 1$.
	\end{lem}

	\begin{proof}
%		\begin{comment}
%		Fix $\delta>0$. Fix $\varepsilon > 0$ and let $N>0$ be such that for $k>N$,
%		$$
%			\E\left| \frac{U^0_k}{k^2} - 2\vartheta_0 \right| < \varepsilon.
%		$$
%		Then, since $U^0$ is independent of the environment,
%		\begin{multline*}
%			\P \left[ \left|\frac{U^0_{n-S_{\nu_n -1}}}{(n-S_{\nu_n -1})^2} - \theta_0 \right| > \delta \right] \leq \P[n-S_{\nu_n -1} < N] \\
%				+ \P\left[\left|\frac{U^0_{n-S_{\nu_n -1}}}{(n-S_{\nu_n -1})^2} - 2\vartheta_0 \right| > \delta, n-S_{\nu_n - 1} \geq N\right] \leq \P[n-S_{\nu_n -1} < N] + \varepsilon/\delta.
%		\end{multline*}
%		Since $\E \xi =\infty$, $n-S_{\nu_n-1} \to^\P \infty$ and it follows that
%		$$
%			\frac{U^0_{n-S_{\nu_n -1}}}{(n-S_{\nu_n -1})^2} \to^\P 2 \vartheta_0.
%		$$
%		Now the claim of the lemma follows from the fact that $S_{\nu_n-1}/n \to \Upsilon(L)$ almost surely and
%	\end{comment}
		 By the merit of  Remark~\ref{rmk:as}, $S_{\nu_n - 1}/n \to \Xi$,  $\P$-almost surely. Secondly, by a standard application of the key renewal theorem \cite[Theorem 2.6.12]{durrett2019probability}, the condition $\EE \xi=\infty$ implies that
	$n-S_{\nu_n - 1} \overset{\P}{\to} \infty$. The claim of the lemma follows from the fact that
		\begin{multline*}
			\frac{U^0_{n-S_{\nu_n -1}} - \Eo U^0_{n-S_{\nu_n -1}}}{ n^2} - (1-\Xi)^2 (2\vartheta_0-1)=\\
			%-\frac{(n-S_{\nu_n -1})^2}{n^2} + (1- \Xi)^2
			%	+ \frac{(n-S_{\nu_n-1})^2}{n^2} \frac{U^0_{n-S_{\nu_n -1}}}{(n-S_{\nu_n -1})^2} - (1- \Xi)^22\vartheta_0.
		 -\bigg( 1 - \frac{S_{\nu_n -1}}{n}\bigg)^2  + (1- \Xi)^2
				+  \bigg( 1 - \frac{S_{\nu_n -1}}{n}\bigg)^2  \frac{U^0_{n-S_{\nu_n -1}}}{(n-S_{\nu_n -1})^2} - (1- \Xi)^22\vartheta_0
		\end{multline*}
	and Lemma \ref{lem:m3}.
%{\red \tt trzeba wyjasnic w jakim sensie sa powyzsze zbieznosci, to nie jest jasne}
	\end{proof}

\subsection{Strong sparsity: $\beta=1$}

	We will now focus on the case when $\beta=1$ and $\E \xi =\infty$.
By Lemmas~\ref{lem:3:reduction},~\ref{lem:5:var1}, \ref{lem:4:nun} and~\ref{lem:5:u0}, it is sufficient to study the quenched behaviour of $\sum_{k=1}^{\nu_n-1} \xi_k^2 (2\vartheta_k-1)$.

	\begin{proof}[Proof of Theorem~\ref{thm:2:m2}]
		Fix $\varepsilon>0$. On the set $\{ |\nu_n - c_n|\leq \varepsilon c_n \}$,
		\begin{equation*}
			\E_\omega \bigg( \sum_{k=c_n+1}^{\nu_n} \xi_k^2(2 \vartheta_k-1) \bigg)^2 =\sum_{k=c_n+1}^{\nu_n}  \xi_k^4 \EE(2\vartheta-1)^2 \leq^{st} C \sum_{k=1}^{\varepsilon c_n} \xi_k^4.
		\end{equation*}
		Observe that
		\begin{equation*}
			 a_{c_n}^{-4}\sum_{k=1}^{\varepsilon c_n} \xi_k^4 =  \varepsilon^{4}(1+o(1))\sum_{k=1}^{\varepsilon c_n} \xi_k^4/a_{\varepsilon c_n}^4.
		\end{equation*}
		Since the sequence on the right hand side is tight in $n$, it follows that
		\begin{equation*}
			a_{c_n}^{-4}\E_\omega \bigg( \sum_{k=c_n+1}^{\nu_n} \xi_k^2(2 \vartheta_k-1) \bigg)^2 \overset{\P}{\to} 0.
		\end{equation*}
		In a similar fashion,
		\begin{equation*}
			a_{c_n}^{-4}\E_\omega \bigg( \sum_{k=\nu_n+1}^{c_n} \xi_k^2(2 \vartheta_k-1) \bigg)^2 \overset{\P}{\to} 0.
		\end{equation*}
		Therefore the weak limit of the quenched law of $(T_n - \E_\omega T_n)/a_{c_n}^2$ will coincide with the limit of
		\begin{equation*}
			\P_\omega \left[ \sum_{k=1}^{c_n} \xi_k^2(2\vartheta_k-1) /a_{c_n}^2 \in \cdot \right].
		\end{equation*}
		The weak limit of the latter is $G(N)$, which follows from the proof of Theorem~\ref{thm:2:m1}.
\end{proof}

\subsection{Strong sparsity: $\beta<1$.}

\begin{proof}[Proof of Theorem~\ref{thm:2:m3}]
Let $\mu_{n,\omega}$ denote the quenched law of $(T_n-\E_\omega T_n)/n^2$. Then
\begin{equation*}
	\mu_{n,\omega} (\cdot) = \P_\omega \left[ \frac{T_{n} - T_{S_{\nu_n-1}} - \E_\omega[T_{n} - T_{S_{\nu_n-1}}]}{ n^2}  +\frac{ T_{S_{\nu_n-1}} -\E_\omega[T_{S_{\nu_{n}-1}}]}{n^2}\in \cdot\right]
\end{equation*}
To treat the second term under the probability we can, similarly as previously, decouple the times that the random walker spends between consecutive $S_k$'s for $k \leq n$. The first part will be
controlled with the help of Lemma~\ref{lem:5:u0}.
%\begin{lem}
%	Assume that~\eqref{eq:2:assumption} and \eqref{eq:2:assumption2} hold true for $\beta < 1$. Then the sequence of random variables
%	$$
%		\Delta_n=\sum_{k=1}^{L^{\leftarrow}(1)d_n} x_k(L_n)^2(2\vartheta_k -1)  - \sum_{k=1}^{\nu_n-1} x_k(L_n)^2(2\vartheta_k -1)
%	$$
%	satisfies $\Delta_n/n^2 \to^\P 0$.
%\end{lem}
%\begin{proof}
%	Fix $\varepsilon >0$. On the set $A_n=\{ |\nu_n-L^{\leftarrow}(1^-)d_n| \leq \varepsilon d_n\}$, $\Delta_n$ is a sum of at most $\varepsilon d_n$ $\P_\omega$-independent random variables and thus
%	\begin{equation*}
%		n^{-4}\E_\omega \left(\Delta_n \right)^2  \mathbbm{1}_{A_n} \leq^{st} \epsilon^{4/\beta} \EE[(2\vartheta-1)^2] \sum_{k=1}^{\varepsilon d_n} \xi_k^4/a_{\epsilon d_n}^4
%	\end{equation*}
%	The claim follows form the tightness of the right hand side.
%\end{proof}
Let $(U^{0}_n, \vartheta_0)$ be, as before, a copy of $(U_n,\vartheta)$ given by the claim of Lemma~\ref{lem:m3} independent of the environment. Then $U_{n-S_{\nu_n -1}}$ has, under $\Po$, the same distribution as the time the
walk spends in $[S_{\nu_n-1},n)$ after reaching $S_{\nu_n-1}$ and before reaching $n$. By Lemma \ref{lem:4:nun} and Lemma \ref{lem:3:reduction} the weak limit of $\mu_{n,\omega}$ is the same as that of 
\begin{equation*}
\bar \mu_{n,\omega} (\cdot) = \P_\omega \left[ \frac{U^0_{n-S_{\nu_n -1}} - \Eo U^0_{n-S_{\nu_n -1}}}{ n^2}  + \frac{ T^r_{S_{\nu_n-1}} -\E_\omega[T^r_{S_{\nu_{n}-1}}]}{ n^2}\in \cdot\right].
\end{equation*}

	Recall the random functions $L_n$ given in~\eqref{eq:5:sub} and that for a c\`adl\`ag function $h$ we denote by $\{ x_k(h), t_k(h)\}$ an arbitrary enumeration of its discontinuities, i.e. $x_k(h) = h(t_k)-h(t_k^-)>0$, where $t_k(h)=t_k$.
	Note that, with $\Upsilon$ given in~\eqref{eq:2:upsilon}, one has
	by the merit of Lemmas~\ref{lem:3:reduction},~\ref{lem:5:var1}, \ref{lem:4:nun} and~\ref{lem:5:u0} that the limit of $\bar \mu_{n,\omega}$
	will coincide with the limit of
	\begin{equation*}
		 F^n(\cdot )=\P_\omega \left[ \frac{a_{d_n}^2}{n^2}(1-\Upsilon(L_{d_n}))^2 (2\vartheta_0-1) + \frac{a_{d_n}^2}{n^2}\sum_{k} x_k(L_{d_n})^2(2\vartheta_k-1){\bf 1}_{\{ L_{d_n}(t_k) <n/a_{d_n} \}}  \in \cdot  \right]
	\end{equation*}
	It is enough to show that $F^n \Rightarrow  F(L)$.
	To achieve that one uses the same approach as in the proof of Theorem~\ref{thm:4:Sn}. Namely by considering for $\varepsilon>0$
	\begin{multline*}
		F_\varepsilon^n (\cdot)= \P_\omega \bigg[  \frac{a_{d_n}^2}{n^2} (1-\Upsilon(L_{d_n}))^2 (2\vartheta_0-1) \\ +  \frac{a_{d_n}^2}{n^2}\sum_{k} x_k(L_{d_n})^2(2\vartheta_k-1) {\bf 1}_{\{ x_k(L_{d_n})>\varepsilon  \}} {\bf 1}_{\{ L_{d_n}(t_k) <n/a_{d_n} \}} \in \cdot  \bigg].
	\end{multline*}
	%as $\varepsilon \to 0$
For fixed $\varepsilon >0$,  $F_\varepsilon^n \to F_\varepsilon^\infty$, where
	\begin{equation*}
			F_\varepsilon^\infty(\cdot) = \P_\omega \left[ (1-\Upsilon(L))^2 (2\vartheta_0-1) + \sum_{k} x_k(L)^2(2\vartheta_k-1) {\bf 1}_{\{ x_k(L)>\varepsilon  \}}
			{\bf 1}_{\{ L_{}(t_k) \leq 1 \}} \in \cdot  \right]
	\end{equation*}	
	since associated point processes converge and  ${a_{d_n}}/{n} \to 1$. Then we show that $F_\varepsilon^\infty \Rightarrow F(L)$ as $\varepsilon \to 0$. We finally prove that~\eqref{eq:4:cond3} also holds in this context and conclude the result.
\end{proof}

\vspace{5mm}

\noindent {\bf Acknowledgement}.
DB nad PD were supported by the National Science Center, Poland (Opus, grant number 2020/39/B/ST1/00209).
AK was supported by the National Science Center, Poland (Opus, grant number 2019/33/B/ST1/00207). 

\bibliographystyle{amsplain}
\bibliography{WQLTforRWSRE}

\providecommand{\bysame}{\leavevmode\hbox to3em{\hrulefill}\thinspace}
\providecommand{\MR}{\relax\ifhmode\unskip\space\fi MR }
\providecommand{\MRhref}[2]{%
  \href{http://www.ams.org/mathscinet-getitem?mr=#1}{#2}
}
\providecommand{\href}[2]{#2}
\begin{thebibliography}{10}

\bibitem{billingsley1999convergence}
P.~Billingsley, \emph{Convergence of probability measures}, second ed., Wiley
  Series in Probability and Statistics: Probability and Statistics, John Wiley
  \& Sons, Inc., New York, 1999, A Wiley-Interscience Publication.

\bibitem{bingham:1987:regular}
N.~H. Bingham, C.~M. Goldie, and J.~L. Teugels, \emph{Regular variation},
  Encyclopedia of Mathematics and its Applications, vol.~27, Cambridge
  University Press, Cambridge, 1987.

\bibitem{buraczewski:2016:power}
D.~Buraczewski, E.~Damek, and T.~Mikosch, \emph{Stochastic models with
  power-law tails}, Springer Series in Operations Research and Financial
  Engineering, Springer, [Cham], 2016, The equation $X=AX+B$.

\bibitem{buraczewski2018precise}
D.~Buraczewski and P.~Dyszewski, \emph{Precise large deviations for random walk
  in random environment}, Electronic Journal of Probability \textbf{23} (2018),
  1--26.

\bibitem{buraczewski:2020:random}
D.~Buraczewski, P.~Dyszewski, A.~Iksanov, and A.~Marynych, \emph{Random walks
  in a strongly sparse random environment}, Stochastic Processes and their
  Applications \textbf{130} (2020), 3990--4027.

\bibitem{buraczewski:2019:random}
D.~Buraczewski, P.~Dyszewski, A.~Iksanov, A.~Marynych, and A.~Roitershtein,
  \emph{Random walks in a moderately sparse random environment}, Electronic
  Journal of Probability \textbf{24} (2019).

\bibitem{dembo1996tail}
A.~Dembo, Y.~Peres, and O.~Zeitouni, \emph{Tail estimates for one-dimensional
  random walk in random environment}, Comm. Math. Phys. \textbf{181} (1996),
  no.~3, 667--683.

\bibitem{Denisov2008}
D.~Denisov, A.~B. Dieker, and V.~Shneer, \emph{Large deviations for random
  walks under subexponentiality: The big-jump domain}, Annals of Probability
  \textbf{36} (2008), 1946--1991.

\bibitem{dolgopyat:goldsheid}
D.~Dolgopyat and I.~Goldsheid, \emph{Quenched limit theorems for nearest
  neighbour random walks in 1{D} random environment}, Comm. Math. Phys.
  \textbf{315} (2012), no.~1, 241--277.

\bibitem{dudley2018real}
R.~M. Dudley, \emph{Real analysis and probability}, CRC Press, 2018.

\bibitem{durrett2019probability}
R.~Durrett, \emph{Probability---theory and examples}, Cambridge Series in
  Statistical and Probabilistic Mathematics, vol.~49, Cambridge University
  Press, Cambridge, 2019.

\bibitem{enriquez:i:spolka}
N.~Enriquez, C.~Sabot, L.~Tournier, and O.~Zindy, \emph{Quenched limits for the
  fluctuations of transient random walks in random environment on {$\Bbb
  Z^1$}}, Ann. Appl. Probab. \textbf{23} (2013), no.~3, 1148--1187.

\bibitem{goldsheid2007simple}
I.~Goldsheid, \emph{Simple transient random walks in one-dimensional random
  environment: the central limit theorem}, Probab. Theory Related Fields
  \textbf{139} (2007), no.~1-2, 41--64.

\bibitem{harris1952first}
T.~E. Harris, \emph{First passage and recurrence distributions}, Transactions
  of the American Mathematical Society \textbf{73} (1952), no.~3, 471--486.

\bibitem{Harrison:Shepp:1981}
J.M Harrison and L.A. Shepp, \emph{On skew brownian motion}, The Annals of
  Probability (1981), 309–313.

\bibitem{kesten1975limit}
H.~Kesten, M.~V. Kozlov, and F.~Spitzer, \emph{A limit law for random walk in a
  random environment}, Compositio Mathematica \textbf{30} (1975), no.~2,
  145--168.

\bibitem{matzavinos:2016:random}
A.~Matzavinos, A.~Roitershtein, and Y.~Seol, \emph{Random walks in a sparse
  random environment}, Electronic Journal of Probability \textbf{21} (2016).

\bibitem{peterson2009quenched}
J.~Peterson, \emph{Quenched limits for transient, ballistic, sub-{G}aussian
  one-dimensional random walk in random environment}, Ann. Inst. Henri
  Poincar\'{e} Probab. Stat. \textbf{45} (2009), no.~3, 685--709.

\bibitem{peterson:2013:weak}
J.~Peterson and G.~Samorodnitsky, \emph{Weak quenched limiting distributions
  for transient one-dimensional random walk in a random environment}, Ann.
  Inst. Henri Poincar\'{e} Probab. Stat. \textbf{49} (2013), no.~3, 722--752.

\bibitem{peterson2009quenched2}
J.~Peterson and O.~Zeitouni, \emph{Quenched limits for transient, zero speed
  one-dimensional random walk in random environment}, The Annals of Probability
  \textbf{37} (2009), no.~1, 143--188.

\bibitem{resnick:2013:extreme}
S.~I. Resnick, \emph{Extreme values, regular variation and point processes},
  Springer New York, 1987.

\bibitem{resnick2007heavy}
\bysame, \emph{Heavy-tail phenomena: probabilistic and statistical modeling},
  Springer Science \& Business Media, 2007.

\bibitem{revuz:2014:continuous}
D.~Revuz and M.~Yor, \emph{Continuous martingales and brownian motion}, 2004.

\bibitem{alili1999asymptotic}
A.~Sma{\"\i}l, \emph{Asymptotic behaviour for random walks in random
  environments}, Journal of applied probability \textbf{36} (1999), no.~2,
  334--349.

\bibitem{solomon1975random}
F.~Solomon, \emph{Random walks in a random environment}, The Annals of
  Probability \textbf{3} (1975), no.~1, 1--31.

\bibitem{henry:2011}
P.~Straka and B.~I. Henry, \emph{Lagging and leading coupled continuous time
  random walks, renewal times and their joint limits}, Stochastic Processes and
  their Applications \textbf{121} (2011), 324--336.

\bibitem{whitt:1971}
W.~Whitt, \emph{Weak convergence of first passage time processes}, 1971,
  pp.~417--422.

\bibitem{Zeitouni:2004:random}
O.~Zeitouni, \emph{Random walks in random environment}, 2004.

\end{thebibliography}

\vspace{1cm}

\footnotesize

\textsc{Dariusz Buraczewski, Piotr Dyszewski and Alicja Kolodziejska}: Mathematical
Institute, University of Wroclaw, 50-384 Wroclaw, Poland\\
\textit{E-mail}: \texttt{dariusz.buraczewski@math.uni.wroc.pl}, \texttt{piotr.dyszewski@math.uni.wroc.pl}, \\ \texttt{alicja.kolodziejska@math.uni.wroc.pl}

\end{document}